\documentclass[11pt]{article}
\usepackage{amsmath,amsfonts,amsthm,amscd,amssymb,graphicx}
\usepackage{subfigure}

\numberwithin{equation}{section}

\usepackage{hyperref}
\usepackage{verbatim} 
\usepackage{color}
\newtheorem{theorem}{Theorem}[section]

\newtheorem{lemma}[theorem]{Lemma}

\newtheorem{proposition}[theorem]{Proposition}

\newtheorem{remark}[theorem]{Remark}
\newtheorem{definition}[theorem]{Definition}

\def\eps{\varepsilon }

\renewcommand{\div}{{\rm div}}

\newcommand{\RR}{\mathbb{R}}

\newcommand{\ZZ}{{\mathbb Z}}

\def\beq{\begin{equation}}
\def\eeq{\end{equation}}
\def\bb1{{1\!\!1}}

\def\cL{\mathcal{L}}
\def\cW{\mathcal{W}}

\def\R{\mbox{Re }}

\def\w{{\omega}}

\def\pt{\partial}

\def\eps{\varepsilon}

\def\triangle{\Delta}
\def\bega{\begin{aligned}}
\def\enda{\end{aligned}}
\def\R{\mathbb{R}^2}

\def\lw{\left}
\def\rw{\right}

\def\R{\mathbb{R}}
\def\wtd{\widetilde}

\def\L{\mathcal{L}}
\def\Z{\mathbb{Z}}
\def\bcase{\begin{cases}}
\def\ecase{\end{cases}}
\def\al{\alpha}
\def\bmx{\begin{bmatrix}}
\def\emx{\end{bmatrix}}

\begin{document}

\title{The inviscid limit of Navier-Stokes equations for locally near boundary analytic data on an exterior circular domain} 
\author{Toan T. Nguyen\footnotemark[1] 
\and
Trinh T. Nguyen\footnotemark[2]
}

\maketitle

\renewcommand{\thefootnote}{\fnsymbol{footnote}}

\footnotetext[1]{Department of Mathematics, Penn State University, State College, PA 16803. Email: nguyen@math.psu.edu. The author is partly supported by the NSF under grant DMS-2054726.}

\footnotetext[2]{Department of Mathematics, University of Southern California, LA, CA 90089.
Email: tnguyen5@usc.edu. The author is partly supported by the AMS-Simons Travel Grant Award.
}

\begin{abstract}

In their classical work \cite{SammartinoCaflisch2}, Caflisch and Sammartino established the inviscid limit and boundary layer expansions of vanishing viscosity solutions to the incompressible Navier-Stokes equations for analytic data on a half-space. It was then subsequently announced in their {\em Comptes rendus} article \cite{SamCaf-ext} that the results can be extended to include analytic data on an exterior circular domain, however the proof appears missing in the literature. The extension to an exterior domain faces a fundamental difficulty that the corresponding linear semigroup may not be contractive in analytic spaces as was the case on the half-space \cite{2N}. In this paper, we resolve this open problem for a much larger class of initial data. The resolution is due to the fact that it suffices to propagate solutions that are analytic only near the boundary, following the framework developed in the recent works that involve the boundary vorticity formulation, the analyticity estimates on the Green function, the adapted geodesic coordinates near a boundary, and the Sobolev-analytic iterative scheme.

\end{abstract}


\section{Introduction}
In this paper, we consider the Navier-Stokes equations with small viscosity $\nu>0$
\beq\label{NS-intro}
\bega 
\pt_t u^\nu+u^\nu\cdot\nabla u^\nu+\nabla p^\nu&=\nu\triangle u^\nu,\\
\nabla\cdot u^\nu&=0,\\
u^\nu|_{\pt\Omega}&=0,
\enda 
\eeq 
on an exterior circular domain $\Omega$ in $\RR^2$, modeling the dynamics of an incompressible fluid around a solid body at a sufficiently high Reynolds number. Of great physical and mathematical interest is the asymptotic behavior of solutions to \eqref{NS-intro} in the small viscosity limit. When $\nu=0$, \eqref{NS-intro} reduces to the Euler equations 
\beq\label{Euler-intro} 
\pt_tu^0+u^0\cdot\nabla u^0+\nabla p^0=0,\qquad \nabla\cdot u^0=0
\eeq
with the non-penetration boundary condition $u^0\cdot n=0$ on the boundary $\pt\Omega$.
Thus, in the limit when $\nu\to 0$, one would formally expect the solutions of the Navier-Stokes equations to converge to $u^0$ in $L^2(\Omega)$ uniformly for a short time, however it remains elusive whether this may be the case. 
Boundary layers appear due to the discrepancy between the boundary conditions in \eqref{NS-intro} and in the limiting model \eqref{Euler-intro}, generating arbitrarily large vorticity near the boundary.  
Kato in his celebrated work \cite{Kato84} shows that the inviscid limit, i.e.  the strong convergence of solutions in the natural energy norm, holds if and only if 
\beq\label{kato-con}
\nu \int_0^T\int_{\{d(x,\pt\Omega)\lesssim \nu\}} \lw|\nabla u^\nu(t)\rw|^2dxdt\to 0\qquad \text{as}\quad \nu\to 0,
\eeq
which implies that the vorticity needs to be controlled quantitatively near the boundary. For general smooth initial data, vorticity can however be very unstable on the boundary that could generate multi-layer solutions at different smaller scales \cite{Grenier-CPAM,Toan-Grenier-APDE}, leading to a larger and larger vorticity than expected, and the inviscid limit problem is therefore unlikely to hold. See, for instance, \cite{Bardos, PC,MM} and the references therein for further discussion. In this paper, we consider smooth data that are analytic locally near the boundary.

\subsection{Previous results}
\underline{When $\Omega$ is the half-space}: In their classical work, Sammartino-Caflisch \cite{SammartinoCaflisch2} established the inviscid limit and Prandtl's boundary layer expansions for analytic data: namely, 
\beq \label{Prandtl}
u^\nu(t,x,y)=u^0(t,x,y)+u^P\lw(t,x,\frac{y}{\sqrt{\nu}}\rw)+o(1)_{L^\infty} ,
\eeq
where the error term $o(1)_{L^\infty}$ is in fact of order $\sqrt \nu$ for such analytic data and thus vanishing in the inviscid limit. The result is extended by Maekawa \cite{Maekawa14} for Sobolev data whose vorticity is compactly supported away from the boundary. Unlike \cite{SammartinoCaflisch2}, Maekawa constructed his solution via the vorticity formulation with a nonlocal boundary condition, which reveals more explicitly the localized interaction between boundary layers and interior solutions. It was this vorticity formulation that leads to a more user-friendly direct proof of the inviscid limit given in \cite{2N} by the authors of the present work, where we in addition devise analytic boundary layer norms, adapted from those introduced in \cite{Toan-Grenier-APDE}, that capture precisely the unbounded vorticity near the boundary. Building upon \cite{2N,Maekawa14}, Kukavica-Vicol-Wang \cite{KVW1} introduced suitable Sobolev-analytic norms that allow to establish the inviscid limit for data that are analytic only near the boundary; see also \cite{Wang} for a similar result in $3D$, \cite{trinh-igor} for the validity of \eqref{Prandtl} for such data, and \cite{GMM1,GMM2} for interesting stability results for data in some Gevrey classes. For Sobolev data, in strong contrast with the analytic case, the Prandtl Ansatz \eqref{Prandtl} is false due to counter-examples given in \cite{Grenier-CPAM,Toan-Grenier-APDE,GrenierNguyen2}.

~\\
\underline{When $\Omega$ is a bounded domain}: There are only few results in the literature that study the inviscid limit problem in fluid domains with a curved boundary. We mention a recent work \cite{Gie} that studies boundary layers in a suitable linearized flow in a general 3D smooth domain and \cite{WW20} which establishes a Prandtl asymptotic expansion in domain with a curved boundary. Very recently, building upon the recent advances including the vorticity formulation revived in \cite{Maekawa14}, the direct proof via the Green function approach developed in \cite{2N}, and the Sobolev-analytic norms introduced in \cite{KVW1}, Bardos-Nguyen-Nguyen-Titi \cite{2N1}
prove the inviscid limit for data that are analytic only near the boundary in a 2D bounded domain.

~\\
\underline{When $\Omega$ is an exterior domain:} In \cite{SamCaf-ext}, Caflisch and Sammartino give a short announcement on obtaining the inviscid limit for analytic data in an exterior circular domain, saving the full proof to be published in one of their listed references, which we are unable to locate. In this paper, we provide the missing proof. We refer the readers to Section \ref{sec-difficulty} where we explain the fundamental difficulty and our main strategy to establish the main result.

\subsection{Boundary vorticity formulation}\label{sec-vortform}

We consider the Navier-Stokes equations posed on the following circular exterior domain
\[
\Omega=\{(x_1,x_2)\in \R^2:\quad x_1^2+x_2^2>1\},
\]
in which for sake of presentation the radius is taken to be one. 
We shall work with the standard polar coordinates $(x_1,x_2) = (r\cos \theta, r\sin \theta)$ for $(r,\theta)\in [1,\infty)\times \mathbb{T}$. Let $e_r = (\cos\theta, \sin \theta)$ and $e_\theta = (-\sin \theta, \cos \theta)$ be the orthogonal frame, and set $(a,b)^\perp = (b,-a)$. We note that 
$$ \nabla = e_r \partial_r + \frac1r e_\theta \partial_\theta, \qquad \Delta = \partial_r^2+\frac{1}{r}\partial_r+\frac{1}{r^2}\partial_\theta^2,$$  
Thus, we write 
$$ u=u_re_r+u_{\theta}e_{\theta}, \qquad \omega = \nabla^\perp \cdot u =  \frac1r \partial_\theta u_r - \frac1r \partial_r (ru_\theta)$$
for velocity and vorticity of the fluid. The Navier-Stokes equations \eqref{NS-intro} can be written in the vorticity formulation as follows:
\beq \label{ext-eq}
\bega
\partial_t\omega-\nu \Delta_{r,\theta}\omega&=- u_r\partial_r\omega - \frac{1}{r}u_{\theta}\partial_\theta \omega
\enda
\eeq
on $[1,\infty)\times \mathbb{T}$, in which $\Delta_{r,\theta} = \partial_r^2+\frac{1}{r}\partial_r+\frac{1}{r^2}\partial_\theta^2$. Making use of the incompressibility condition, we introduce the stream function $\psi = \psi(r,\theta)$ defined through $u = \nabla^\perp \psi$, or equivalently
\begin{equation}\label{def-u} u _r = \frac1r\pt_\theta\psi, \qquad u_\theta = -\pt_r\psi .\end{equation}
By definition, the stream function solves the elliptic problem 
\beq\label{def-stream}
\begin{cases}
&\Delta_{r,\theta}\psi=\w\\
&\psi_{\vert_{r=1}}=0
\end{cases}
\eeq
whose solutions can be constructed explicitly through the Green function; see Section \ref{ellip-sec}. 

Therefore, the Navier-Stokes equation problem \eqref{NS-intro} reduces to study the scalar vorticity equation \eqref{ext-eq} on $[1,\infty)\times \mathbb{T}$, where the velocity is constructed through the Biot-Savart law \eqref{def-u}-\eqref{def-stream}. As for the no-slip boundary condition, $u_r =0$ follows from the condition $\psi=0$ on the boundary, while $u_\theta=0$ is a direct consequence of the following imposed condition 
$$ \partial_t u_\theta = 0$$
from which we derive the boundary condition on vorticity $\omega$. This formulation was introduced and developed in \cite{Anderson,Maekawa14}. See also \cite{2N,2N1}. 
 Indeed, by construction, we compute 
\beq \label{h-BC}
\begin{aligned}
0 = \partial_t u_\theta = -\partial_r \Delta^{-1}\partial_t \omega 
&= -\partial_r [\Delta^{-1} (\nu \Delta \omega - u \cdot \nabla \omega)]
\end{aligned}\eeq
on the boundary. This yields the following boundary condition for vorticity 
\beq\label{BC-vor}
\nu (\partial_r + N)\omega_{\vert_{r=1}} = [\partial_r \Delta^{-1} ( u \cdot \nabla \omega)]
_{\vert_{r=1}}
\eeq
where $N$ denotes the Dirichlet-Neumann operator on $\Omega$, which will be detailed in Section \ref{sec-NSE}.

\subsection{Main result}

Our main result is to establish a uniform bound on the vorticity and the inviscid limit of solutions to the Navier-Stokes problems for initial data whose vorticity is locally analytic near the boundary $r=1$. Precisely, 

\begin{definition} Let $\delta_0>0$ and $p\ge 1$. An $L^p$  function $f(r)$ defined on $[1,1+\delta_0]$ is said to be locally analytic near the boundary $r=1$ if it can be extended analytically to the pencil-like complex domain 
\[\bega 
R_{\rho}=& \Big\{r\in \mathbb{C}:\quad 1\le \Re r\le 1+\delta_0,\quad |\Im r|\le \rho(\Re r-1) \Big\}
\enda 
\]
for some positive analyticity radius $\rho$ with a finite norm $\|f\|_{L ^p_\rho}=\sup_{0 \le \eta <\rho}\| f \|_{L^p(\pt R_\eta)} $. 
\end{definition}

Note that a locally near boundary analytic function needs not to be analytic on the boundary, but only has bounded derivatives $(r-1)\partial_r$. Our main result is stated as follows:

\begin{theorem}\label{theo-main} Consider the vorticity equation \eqref{ext-eq} on $[1,\infty)\times \mathbb{T}$ with the boundary condition \eqref{BC-vor} and the Biot-Savart law \eqref{def-u}-\eqref{def-stream}. Assume that initial vorticity $\w^\nu_0(r,\theta)$ has Sobolev regularity $r^2\w^\nu_0 \in H^3([1,\infty)\times \mathbb{T})$, and its Fourier coefficients $\w^\nu_{0,n}(r)$ with respect to variable $\theta$ are locally analytic near the boundary and satisfy 
\begin{equation}\label{analytic-assmp}
\sum_{n\in \mathbb Z}e^{\eps_0|n|} \|\w^\nu_{0,n}(r)\|_{L^1_{\rho_0}} < \infty
\end{equation}
uniformly in $\nu$, for some positive constants $\epsilon_0, \rho_0$. Then, there is a positive time $T$, independent of $\nu$, so that the Navier-Stokes vorticity satisfies
\begin{equation}\label{vorticity-bd}
 \|\w^\nu(t)\|_{L^\infty(\partial \Omega)} \le C_0(\nu t)^{-1/2}
\end{equation}
for $t\in (0,T]$, and the inviscid limit holds: that is, there exists a unique limiting solution $u^0$ that solves the corresponding solution to Euler equations \eqref{Euler-intro} so that 
\beq\label{inv-lim}
\sup_{0\le t\le T}\|u^\nu-u^0\|_{L^2(\Omega)}\to 0\qquad \text{as}\quad \nu\to 0 .
\eeq
\end{theorem}
\begin{remark} 
If we replace the assumption \eqref{analytic-assmp} by a stronger assumption
\[
\sum_{n\in \mathbb Z}e^{\eps_0|n|} \|\w^\nu_{0,n}(r)\|_{L^\infty_{\rho_0}} < \infty
\]
then \eqref{vorticity-bd} can be improved to 
$\sup_{0\le t\le T}\|\w^\nu(t)\|_{L^\infty(\partial \Omega)} \le C_0\nu^{-1/2}.
$
\end{remark}

The inviscid limit is a direct consequence of the boundary vorticity estimates \eqref{vorticity-bd}, which is optimal in view of the boundary layer expansion \eqref{Prandtl} as predicted by Prandtl and justified for analytic data  \cite{SammartinoCaflisch2}. The assumption \eqref{analytic-assmp} holds in particular for data whose vorticity vanishes near the boundary, and the theorem thus recovers the result by Maekawa \cite{Maekawa14} to the case of exterior circular domains. We stress that the near boundary analyticity assumption \eqref{analytic-assmp} is necessary for the vorticity bound \eqref{vorticity-bd} to hold, since otherwise the presence of near boundary high frequency will generate boundary viscous sublayers \cite{Toan-Grenier-APDE}, whose vorticity is proven to reach order $\nu^{-3/4}$, much larger than the Prandtl's classical prediction of order $\nu^{-1/2}$. In general, much worse and more complex structure of boundary vorticity is expected; see  \cite{Grenier-CPAM,GGN1,Toan-Grenier-APDE,GrenierNguyen2} for further discussion.

\subsection{Difficulties and main ideas}\label{sec-difficulty}
Let us discuss the difficulties in proving the inviscid limit when the domain is an exterior circular disk. In view of the previous works \cite{2N,2N1}, there are several difficulties that one has to overcome in the present setting. Namely, the framework relies on the semigroup of the linear Stokes problem, treating the nonlinearity as a perturbation in the Duhamel representation. For the nonlinear iterative scheme to work, it is crucial that the semigroup is contractive in the function spaces under consideration, namely analytic spaces; see Proposition 3.1\footnote{We wish to point out a misprint in \cite[Proposition 3.1]{2N} where the third estimate on the trace semigroup in the boundary layer norm should read
$$ ||| \Gamma(\nu (t-s) )g |||_{\rho,\sigma,\delta(t),k} \lesssim \sqrt{\frac{t}{t-s}}|||g|||_{\rho,k} + \sqrt \nu |||g|||_{\rho,k+1}.$$  
Namely, the last term with one loss of derivatives on the boundary was missing! Note however this is harmless in \cite{2N}, since the estimates were used only to propagate the boundary layer norms {\em after closing} the nonlinear iteration with $L^1$ analytic norms where no loss of derivatives is present on the trace estimates; see the analysis in Section 4.2 of that same paper. 
}
in \cite{2N}. However, the contraction in analytic spaces is open for the linear Stokes problem on the exterior domain. Precisely, we are led to study the following Stokes problem 
\beq\label{Stokesmae}
\left \{\begin{aligned}
\pt_t \w-\nu \lw(\pt_r^2+\frac 1 r \pt_r+\frac{1}{r^2}\pt_\theta^2\rw)\w&=0\\
(\pt_r+|\pt_\theta|)\w_{\vert_{r=1}}&=0
\end{aligned}
\right.\eeq
whose resolvent kernel and Green kernel can be easily constructed. 
Deriving the analytic estimates on the Green function and the semigroup uniformly both in time and in the small viscosity limit however appears an impossible task. Indeed,  following \cite{2N} and working with the Laplace-Fourier transform variables $(\zeta,n)$ associated with $(t,\theta)$, the Green kernel for the resolvent problem 
is of the form 
\[
G_\zeta (r,r')=\frac{1}{W(I_n,K_n)(\mu r')}\frac{I_{n-1}(\mu)}{K_{n-1}(\mu)}K_n(\mu r)K_n(\mu r')+\begin{cases}\frac{I_n(\mu r)K_n(\mu r')}{W(I_n,K_n)(\mu r')}\qquad &\text{if}\quad r<r',\\
\frac{I_n(\mu r')K_n(\mu r)}{W(I_n,K_n)(\mu r')}\qquad &\text{if}\quad r>r',
\\\end{cases}
\]
with $\mu=\sqrt{\frac \zeta \nu}$, where the functions $K_n(z)$ and $I_n(z)$ are modified Bessel functions with complex value $z\in \mathbb C$ (e.g., \cite{Mae-ext}), with $W(I_n,K_n)$ being the Wronskian determinant. The temporal Green function is then defined by taking the inverse Laplace transform in $t$ of the kernel $G_\zeta (r,r')$. Unfortunately, the available pointwise bounds and asymptotic expansions of the modified Bessel functions are given only in the regime for 
 \begin{itemize} \item{fixed $n$, large $r$} 
 \item{or fixed $r$, large $n$,
 }
 \end{itemize}
but not when both $n,r$ are sufficiently large and $\nu$ is sufficiently small. As a consequence, the propagation of uniform semigroup estimates on analytic spaces remains open, and therefore the pointwise Green function approach developed in \cite{2N} does not apply directly. 

We overcome the issue by working with functions that are required to be analytic only near the boundary, see Theorem \ref{theo-Stokes}. Effectively, this only requires analytic estimates of the Green function near the boundary, which is available from the half-space result \cite{2N}. Precisely, close to the boundary $r=1$, we write  
 \[
\pt_r^2+\frac{1}{r}\pt_r+\frac{1}{r^2}\pt_\theta^2=(\pt_r^2+\pt_\theta^2)+\frac {1}{r}\pt_r+\lw(\frac{1}{r^2}-1\rw)\pt_\theta^2
\]
and using the half-space Green kernel for the operator $\pt_r^2+\pt_\theta^2$, treating the remaining terms as a perturbation. Importantly, we note that the last term experiences two a loss of two derivatives and is thus a perturbation only when $r$ is sufficiently close to $1$. See Section \ref{sec3} where we establish the semigroup estimates for the Stokes problem in Sobolev-analytic spaces.    

Finally, unlike the treatment in \cite{2N1}, we need to estimate the solution in the unbounded region and therefore a careful norm with suitable decay is needed. Our vorticity $\omega(r,\theta)$ decays like $r^{-2}$ away from the boundary. 

\section{Scaled equations and locally analytic spaces}

\subsection{Navier-Stokes equations in the rescaled variables}\label{sec-NSE}

To take advantage of localization near the boundary, we introduce a change of variables 
\[
x=\lambda^{-1}\theta,\quad y=\lambda^{-1}(r-1), \quad \tau =\lambda^{-2} t
\]
for some small parameter $\lambda>0$, and define the function $w$ such that 
\beq\label{w-omega}
w(\tau ,x,y)=\w(t,\theta,r)=\w(\lambda^2 \tau , \lambda x, 1+\lambda y)
\eeq
for $x\in \mathbb{T}_{2\pi/\lambda}$ and $y\in \RR_+$. By a direct calculation, we have 
\[\bega
\triangle_{r,\theta} &=\pt_r^2+\frac{1}{r}\pt_r+\frac{1}{r^2}\pt_\theta^2
\\&=\lambda^{-2}\lw(\pt_x^2+\pt_y^2+\lambda a(y)\pt_y+\lambda b(y)\pt_x^2\rw)\\
&=\lambda^{-2}\lw(\triangle_{x,y}+\lambda L\rw)
\enda 
\]
where $\Delta_{x,y} = \pt_x^2+\pt_y^2$ (and hereafter, we simply write $\Delta = \Delta_{x,y}$), 
\begin{equation}\label{def-LL}
L=a(y)\pt_y+b(y)\pt_x^2, \qquad 
a(y)=\frac{1}{1+\lambda y},\qquad b(y)=\frac{y(2+\lambda y)}{(1+\lambda y)^2}.
\end{equation}
From \eqref{ext-eq}, the scaled vorticity $w$ 
satisfies 
\beq\label{new-eq}
\bega 
&(\pt_{\tau }-\nu\triangle )w=\nu \lambda Lw+ B(\psi,w)
\enda 
\eeq
where 
\begin{equation}\label{def-BBB}
B(\psi,w)=-a(y)\nabla^\perp\psi\cdot\nabla w, \qquad \nabla^\perp=(\pt_y,-\pt_x).
\end{equation}
Similarly, abusing the same notation, the scaled stream function $\psi$ solves 
\beq\label{def-streamscaled}
\begin{cases}
&(\triangle+\lambda L)\psi=\lambda^2 w,\\
&\psi|_{y=0}=0.
\end{cases}
\eeq
We next derive a boundary condition for $w$. As mentioned in Section \ref{sec-vortform}, we impose $\partial_\tau u_\theta =0$, which gives $\pt_{\tau }\pt_y\psi|_{y=0}=0$ and so
\[
\bega 
\pt_y(\triangle+\lambda L)^{-1}\pt_{\tau }w|_{y=0}&=0.
\enda\]
Using the vorticity equation \eqref{new-eq}, we get 
\beq \label{bdr-cond1}
\bega 
\pt_y(\triangle+\lambda L)^{-1}\lw(\nu(\triangle+\lambda L)w+B(\psi,w)\rw)|_{y=0}=0.
\enda 
\eeq 
Let $w^\star$ solves
\beq\label{w-star-eq}
(\triangle+\lambda L)w^\star=0,\qquad w^\star|_{y=0}=w|_{y=0}.
\eeq
Then \eqref{bdr-cond1} becomes 
\[
\nu \pt_y (w-w^\star)|_{y=0}=-\pt_y(\triangle+\lambda L)^{-1}\lw(B(\psi,w)\rw)|_{y=0}
\]
Defining $Nw=-\pt_yw^\star|_{y=0}$, which is the classical Dirichlet-to-Neumann operator, we obtain the boundary condition for the vorticity  
\beq\label{bdr1}
\nu(\pt_y+N)w|_{y=0}=-\pt_y(\triangle+\lambda L)^{-1}\lw(B(\psi,w)\rw)|_{y=0}
\eeq
In this paper, for any function $f$ depending on $x\in \mathbb{T}_{2\pi/\lambda}$, we denote $f_\al$ to be the Fourier coefficient of $f$ in the frequency $\al\in \lambda \mathbb Z$, and $f_n$ to be the Fourier coefficient of $f$ in the original variable $\theta\in\mathbb{T}_{2\pi}$ where $n\in\mathbb Z$. We prove the following lemma regarding the Dirichlet-Neumann operator in the new variables:
\begin{lemma} \label{N-def}
The operator $Nw_\al$ can be written as
\[
Nw_\al=|\al|w_\al(0)+\lambda \int_0^\infty  \lw(w_\al(0) L_\al(e^{-|\al|y})+L_\al\wtd w_\al^\star\rw)dy
\]
where $\wtd w_\al^\star$ solves the elliptic problem 
\beq\label{w-star-eq}
\begin{cases}
&(\pt_y^2-\al^2)\wtd w_\al^\star=-\lambda L_\al \lw(w_\al(0)e^{-|\al|y}\rw)-\lambda L_\al \wtd w^\star_\al\\
&\qquad \wtd w_\al^\star|_{y=0}=0
\end{cases}
\eeq
and $L_\al=a(y)\pt_y-\al^2 b(y)$ is the linear operator acting on the frequency $\al$ of L.
\end{lemma}
\begin{proof}
We recall the definition of $w^\star$ in \eqref{w-star-eq}. Taking Fourier in $x$, we obtain 
\[
(\pt_y^2-\al^2)w_\al^\star+\lambda L_\al w_\al^\star=0,\qquad w_\al^\star=w_\al(0).
\]
Let $\wtd w_\al^\star=w_\al^\star(y)-w_\al(0)e^{-|\al|y}$, then 
\[
Nw_\al=-\pt_y w_\al^\star|_{y=0}=-\pt_y\lw(\wtd w_\al^\star+w_\al(0)e^{-|\al|y}\rw)|_{y=0}=-\pt_y\wtd w^\star_\al(0)+|\al|w_\al(0).
\]
we have 
\[
(\pt_y^2-\al^2)\wtd w_\al^\star=-\lambda L_\al \lw(w_\al(0)e^{-|\al|y}\rw)-\lambda L_\al \wtd w^\star_\al,\qquad \wtd w_\al^\star|_{y=0}=0.
\]
By a direct calculation, we have 
\[
\pt_y\wtd w_\al^\star(0)=\int_0^\infty e^{-|\al|y}\lambda \lw(w_\al(0) L_\al(e^{-|\al|y})+L_\al\wtd w_\al^\star\rw)dy.
\]
The proof is complete.
\end{proof} 

\subsection{The half-space problem}

To summarize, we have reduced the Navier-Stokes equations on the exterior disk to the following problem on the half-line $y\ge 0$: for each spatial frequency $\alpha \in \lambda \ZZ$, 
\beq\label{new-eqa}
\bega 
&(\pt_{\tau } -\nu\Delta_\alpha )w=\nu \lambda L_\alpha w+ B_\alpha(\psi,w)
\enda 
\eeq
with notation $\Delta_\alpha = \partial_y^2 - \alpha^2$, $L_\al=a(y)\pt_y-\al^2 b(y)$, and the following nonlocal boundary condition 
\beq\label{g-al}
\bega 
\nu(\pt_y+|\al|)w_\al|_{y=0}=&-\nu \lambda  \int_0^\infty e^{-|\al|y} \lw(w_\al(0) L_\al(e^{-|\al|y})+L_\al\wtd w_\al^\star\rw)dy\\
&-\pt_y(\triangle_\alpha+\lambda L_\alpha)^{-1}\lw(B_\al(\psi,w)\rw)|_{y=0}
\enda 
\eeq
where $B(\psi,w)$ and $\psi$ are defined as in \eqref{def-BBB}-\eqref{def-streamscaled}.

\subsection{Near boundary analytic spaces}\label{sec-norm}
In this section, we introduce the near boundary analytic norms to control the near boundary analyticity and the Sobolev regularity of vorticity. These norms are an adaptation from those that were introduced and developed in \cite{SammartinoCaflisch2,2N,2N1,KVW1}. 

Precisely, let $\delta_0>0$ be the size of the analytic domain for our solution near the
boundary. Throughout the paper, we fix $\rho_0\ge \delta_0$, and take $\rho\in(0,\rho_0)$. We define the complex domain
\[\bega 
\Omega_{\rho}=&\Big \{y\in \mathbb{C}:\quad 0\le \Re y\le \delta_0,\quad |\Im y|\le \rho\Re y\Big\}\\
&\bigcup \Big\{y\in \mathbb{C}:\quad \delta_0\le \Re y\le \delta_0+\rho,\quad |\Im y|\le \delta_0+\rho-\Re y\Big\}.
\enda 
\]
We note that the domain $\Omega_\rho$ only contains $y$ with $0\le \Re y\le \delta_0+\rho$. 
For a complex valued function  $f$ defined on $\Omega_\rho$, let
\[
\|f\|_{L ^1_\rho}=\sup_{0 \le \eta <\rho}\| f \|_{L^1(\pt\Omega_\eta)} , \qquad \|f\|_{L ^\infty_\rho}=\sup_{0 \le \eta <\rho}\| f \|_{L^\infty(\pt\Omega_\eta)}
\]
where the integration is taken over the two directed paths along the boundary of the
domain $\Omega_{\eta}$.
Now for an analytic function $f(x,y)$ defined on $(x,y)\in \mathbb{T}_{2\pi/\lambda}\times \Omega_\rho$, we define
\beq\label{normY-mu}
\begin{aligned}
\|f\|_{\mathcal L^1_\rho} &=\sum_{\alpha\in \lambda \mathbb{Z}}\lw\|e^{\eps_0(\delta_0+\rho- \Re y)|\alpha|}f_\alpha\rw\|_{L^1_\rho} ,
\\
 \|f\|_{\mathcal L^\infty_\rho} &=\sum_{\alpha\in \lambda \mathbb{Z}}\lw\|e^{\eps_0(\delta_0+\rho- \Re  y)|\alpha|}f_\alpha\rw\|_{L^\infty_\rho}.
 \end{aligned}
\eeq
The function spaces $\cL^1_\rho$ and $\cL_\rho^\infty$ are to control the scaled vorticity and velocity, respectively. We stress that the analyticity weight is identically zero on $\Re y \ge \delta_0+\rho$. For convenience, we also introduce the following analytic norms
\begin{equation}\label{def-W1k}\| f\|_{\cW_\rho^{k,p}} =  \sum_{i+j\le k}\|\pt_{x}^i (y\pt_{y})^j f\|_{\cL^p_\rho}
\end{equation}
for $k\ge 0$ and $p=1, \infty$. 
The above definition also applies for a function $g$ defined on the domain $\mathbb{T}_{2\pi/\lambda}$. Namely,
\beq\label{H-norm}
\|g\|_{\mathcal H_\rho^{k}}=\sum_{\al\in \lambda \Z}|\al|^k e^{\eps_0(\delta_0+\rho)|\al|}|g_\al|.
\eeq
For convenience, we also write 
\[
\|D_{x,y}^kf\|_X=\sum_{i+j\le k}\|\pt_x^i\pt_y^jf\|_X
\]
where $X$ is a function space. 
We recall the following simple algebra. 
\begin{lemma} There hold
\begin{equation}\label{Y-algebra} \| fg \|_{\cL^1_\rho} \le \| f\|_{\cL^\infty_\rho } \| g\|_{\cL^1_\rho}
\end{equation}
and for any $0<\rho'<\rho$,
\begin{equation}\label{Y-derivative}
\| \partial_{x} f \|_{\cL^1_{\rho'}} + \| y \partial_{y} f\|_{\cL^1_{\rho'}} \lesssim \frac{1}{\rho - \rho'} \| f\|_{\cL^1_\rho}.
\end{equation}
\end{lemma}

\begin{proof} The proof is direct; see \cite{2N1,2N}. 
\end{proof}
We also have the following lemma, which will be useful in controlling the velocity in the intermediate region in Section \ref{inter-sec}. We note that in the lemma below, we only give the real pointwise bounds in $L^\infty$ norm on the real line.
\begin{lemma}\label{middle}
Let $f=f(x,y)$ be analytic in $\mathbb{T}_{2\pi/\lambda}\times \Omega_\rho$ where the analyticity radius $\rho\ge \frac{\delta_0}{4}$. Then for any $\delta_1<\delta_2<\delta_0$ and $k\ge 0$, we have 
\[\bega 
\|D_{x,y}^k f\|_{L^\infty(\delta_1\le y\le \delta_2)}&\lesssim \|f\|_{\L^1_\rho} .
\enda 
\]
\end{lemma}

\begin{proof} We first prove the bound for $\pt_x^k f$. Since $\pt_x^k f=\sum_\al e^{i\al x}(i\al)^k f_\al$, we have 
\[
\|\pt_x^kf\|_{L^\infty( \delta_1\le y\le \delta_2)}\lesssim \sum_{\al} |\al|^k \|yf_\al(y)\|_{L^\infty(\delta_1\le y\le \delta_2)},
\]
noting $y\ge \delta_1$. Now for any $y\le \delta_2$, we have
\beq\label{yf}
\bega
|yf_\al(y)|&=\lw|\int_0^y \pt_z(zf_\al(z))dz\rw|\le \int_0^{\delta_2} |f_\al(z)|dz+\int_0^{\delta_2} |z\pt_z f_\al(z)|dz. 
\enda
\eeq
For the first integral, we have
\beq\label{yf1}
\int_0^{\delta_2}|f_\al(z)|dz=\int_0^{\delta_2}e^{-\eps_0(\delta_0+\rho-z)|\al|}e^{\eps_0(\delta_0+\rho-z)|\al|}|f_\al(z)|dz\le e^{-\eps_0(\delta_0-\delta_2)|\al|}\|f_\al\|_{\L^1_\rho}.
\eeq
For the second integral, we first use the estimate \eqref{Y-derivative} to get
\beq\label{yf2}
\int_0^{\delta_2}|z\pt_z f_\al(z)|dz\lesssim \frac{1}{\delta_0+\frac{\delta_0}{8}-\delta_2}\|f_\al\|_{\L^1_{\delta_0/8}}\lesssim e^{-\eps_0|\al|(\delta_0/8)} \|f_\al\|_{\L^1_\rho},
\eeq
where in the last inequality, we have used the fact that 
\[
e^{\eps_0|\al|(\delta_0+\delta_0/8-\Re z)}\le e^{-\eps_0|\al|\delta_0/8} e^{\eps_0|\al|(\delta_0+\rho-\Re z)}\qquad \text{for}\quad \Re z\le \delta_0\quad \text{and}\quad \rho\ge \delta_0/4.
\]
Combining the inequalities \eqref{yf}, \eqref{yf1} and \eqref{yf2}, we get 
\[
|yf_\al(y)|\lesssim \lw( e^{-\eps_0(\delta_0-\delta_2)|\al|}+e^{-\eps_0|\al|(\delta_0/8)}\rw)\|f_\al\|_{\L^1_\rho}
\]
The proof for $\pt_x^k f$ is complete, by multiplying both sides of the above inequality by $|\al|^k$ and summing all over $\al$.
Similarly, we compute 
\[\bega
\|\pt_y f\|_{L^\infty(\delta_1\le y\le \delta_0)}&\lesssim \|y^2\pt_yf\|_{L^\infty(\delta_1\le y\le \delta_0)}\le \sum_{\al}\|y^2\pt_yf_\al\|_{L^\infty(\delta_1\le y\le\delta_0)}\\
&\lesssim \sum_{\al} \|y\pt_yf_\al\|_{L^1(y\le \delta_0)}+\sum_{\al}\|y^2\pt_y^2f_\al\|_{L^1(y\le \delta_0)}\\
&\lesssim \|f\|_{\L^1_{\delta_0/4}}\le \|f\|_{\L^1_\rho}.
\enda 
\]
where we use the Cauchy estimate and the fact that $\rho\ge \frac{\delta_0}{4}$. The estimates on higher derivatives follow similarly. 
\end{proof}

\section{The Stokes problem}\label{sec3}
In this section, we consider the Stokes problem in the exterior domain, written in the rescaled geodesic coordinates:
\beq\label{st-ext1}
\bega
(\pt_\tau-\nu\triangle - \nu \lambda L)w&= f\\
\nu(\pt_z+N)w|_{z=0}&=g
\enda 
\eeq
where $L=a(z)\pt_z+b(z)\pt_x^2$ is the linear operator defined in \eqref{def-LL} and $N$ is the Dirichlet-Neumann operator. The main result of this section is to provide uniform estimates on the solution of \eqref{st-ext1} in the Sobolev-analytic spaces. Precisely, we have the following theorem. 

\begin{theorem}\label{theo-Stokes}
Let $e^{\nu t S}$ be the semigroup of the linear Stokes problem \eqref{st-ext1}, and let $\Gamma^S(\nu t)$ be its trace on the boundary. Fix any finite time $T$.
Then, for sufficiently small $\lambda$, and for any $0\le t\le T$, $\rho>0$, and $k\ge 0$, there hold
\begin{equation}\label{exp-realSW}
\begin{aligned}
\| e^{\nu t S} w_0\|_{\cW^{k,1}_\rho} &\le  C_0 \| w_0\|_{\mathcal W_\rho^{k,1}}+\|y^2D_{x,y}^{k+1} w_0\|_{L^2(y\ge \delta_0+\rho)}
\\
\| \Gamma^S(\nu t)g_b\|_{\cW^{k,1}_\rho} &\le C_0 \|g_b\|_{\mathcal H_\rho^{k}}
\end{aligned}\end{equation}
uniformly in the inviscid limit, where $\|\cdot\|_{\mathcal W_\rho^{k,1}},\|\cdot\|_{\mathcal H_\rho^{k}}$ are near boundary analytic norms defined in \eqref{def-W1k} and \eqref{H-norm}, respectively. 
\end{theorem}

The proof relies on the analytic estimates for solutions of the Stokes problem on the half-space. Indeed, in view of Lemma \ref{N-def}, 
we can rewrite the boundary condition in \eqref{st-ext1} as follows:
\[
\nu(\pt_z+|\al|)w_\al|_{z=0}= g_\al + h_\al
\]
where 
\beq\label{h-al}
h_\al=-\lambda\nu \int_0^\infty\lw(w_\al(0)L_\al(e^{-|\al|y})+L_\al\wtd w_\al^\star\rw)dy
\eeq
and $\wtd w_\al^\star$ solves the elliptic problem \eqref{w-star-eq}. Therefore, 
we obtain the following Duhamel principle for solution of \eqref{st-ext1},
\beq\label{duh}
\bega
w(\tau)=&e^{\nu \tau B}w_0+\int_0^\tau e^{\nu (\tau-s)B}(\nu \lambda L w)(s)ds+\int_0^\tau e^{\nu(\tau-s)B}f(s)ds\\
&+\int_0^\tau \Gamma(\nu(\tau-s))g(s)ds+\int_0^\tau \Gamma(\nu(\tau-s))h(s)ds,
\enda 
\eeq
where $e^{\nu \tau B}$ is the Stokes semigroup on the half-space and $\Gamma(\nu {\tau })$ denotes its trace on the boundary. To estimate each term on the right hand side, we first recall the following results from \cite{2N,2N1}, which give the Duhamel formula, Green functions estimates, and semigroup bounds for the Stokes problem on the half-space in the vorticity boundary condition.

\begin{theorem}[\cite{2N,2N1}]
Consider the Stokes problem 
\beq\label{st-hlf}
\bega
(\pt_\tau-\nu\triangle)W&=F\\
\nu(\pt_z+|\pt_x|)W|_{z=0}&=g_b\\
W|_{\tau=0}&=W_0
\enda
\eeq
on the analytic-Sobolev domain $(x,y)\in \mathbb{T}_{2\pi/\lambda}\times \{ \Omega_\rho\cup\{y\ge \delta_0+\rho\} \}$.
The solution to \eqref{st-hlf} can be written as 
\[
W(\tau)=e^{\nu \tau B}W_0+\int_0^{\tau}e^{\nu(\tau-s)B}F(s)ds+\int_0^\tau \Gamma(\nu(\tau-s)) g_b(s)ds
\]
where $e^{\nu \tau B}$ is the Stokes semigroup of the problem \eqref{st-hlf} and $\Gamma(\nu {\tau })$ denotes its trace on the boundary. 
Moreover, there hold the following semigroup estimates
\beq\label{semigroup}
\bega
\|e^{\nu \tau B}W_0\|_{\mathcal W_\rho^{k,1}} &\lesssim \|W_0\|_{\mathcal W_\rho^{k,1}}+\|yD^{k+1}_{x,y} W_0\|_{L^2(y\ge \delta_0+\rho)},\\
\|e^{\nu(\tau-s)B}F(s)\|_{\mathcal W_\rho^{k,1}} &\lesssim \|F(s)\|_{\mathcal W_\rho^{k,1}}+\|yD^{k+1}_{x,y} F(s)\|_{L^2(y\ge \delta_0+\rho)},\\
\|\Gamma(\nu(\tau-s)) g_b(s)\|_{\mathcal W_\rho^{k,1}} &\lesssim \|g_b(s)\|_{\mathcal H_\rho^{k}},
\enda 
\eeq
where $\|\cdot\|_{\mathcal W_\rho^{k,1}},\|\cdot\|_{\mathcal H_\rho^{k}}$ are analytic norms defined in \eqref{def-W1k} and \eqref{H-norm}. 

In addition, the Fourier coefficients $e^{\nu (\tau-s)B_\alpha}$ of the semigroup $e^{\nu (\tau-s)B}$ have a Green kernel representation $G_\al(\tau,y,z)$, in the sense that, for any $\tau \ge 0$, one has
\[
e^{\nu \tau B_\alpha}F_\al(\tau,y)(z)=\int_0^\infty G_\al(\tau,y,z)F_\al (\tau,y)dy
\]
with the decomposition $G_\al(\tau,y,z)=H_\al(\tau,y,z)+R_\al(\tau,y,z)$, in which $H_\alpha({\tau },y;z) $ is exactly the one-dimensional heat kernel with the homogenous Neumann boundary condition and $R_\alpha({\tau },y;z)$ is the residual kernel due to the boundary condition. Precisely, there hold
\beq\label{Stokes-green}
\begin{aligned}
H_\alpha({\tau },y;z) & = \frac{1}{\sqrt{\nu {\tau }}} \Big( e^{-\frac{|y-z|^{2}}{4\nu {\tau }}} +  e^{-\frac{|y+z|^{2}}{4\nu {\tau }}} \Big) e^{-\alpha^{2}\nu {\tau }}, 
\\ | \partial_z^k R_\alpha (\tau,y;z)| &\lesssim \mu_f^{k+1} e^{-\theta_0\mu_f |y+z|} +  (\nu t)^{-\frac{k+1}{2}}e^{-\theta_0\frac{|y+z|^{2}}{\nu {\tau }}}  e^{-\frac18 \alpha^{2}\nu \tau} ,
\end{aligned}
\eeq
for $y,z\ge 0$, $k\ge 0$, and for some $\theta_0>0$ and for $\mu_f = |\alpha| + \frac{1}{\sqrt \nu}$. 
\end{theorem}

We now estimate each term on the right of \eqref{duh}. The terms involving initial data, $f$, and $g$ are already estimated using the semigroup estimates in \eqref{semigroup}. We now estimate the second and last terms on the right of \eqref{duh}. We start with the linear term $\nu\lambda Lw$. 

\begin{proposition}\label{nuL}
Let $L=a(y)\pt_y+b(y)\pt_x^2$. There holds, for $0\le k\le 2$:
\[\bega 
\lw\|\int_0^{\tau } e^{\nu(\tau -s)B}(\nu \lambda Lw)ds\rw\|_{\mathcal W_\rho^{k,1}}
&\le C \lambda  \sup_{0\le s\le \tau }\lw(\|w\|_{\mathcal W^{k,1}_\rho}+\|y^2D^{k+1}_{x,y}w\|_{L^2(y\ge \delta_0+\rho)}\rw)\\
&\quad+C\lambda \nu \int_0^{\tau }\|w(s)\|_{\mathcal W^{k+1,1}_\rho}ds
+C\lambda \nu  \int_0^{\tau } \|w(s)|_{\{z=0\}}\|_{\mathcal H_\rho^{k}}ds .
\enda 
\]
Here, the constant $C$ is independent of $\lambda$.
\end{proposition}
\begin{proof}
We focus on the case when $k=0$; the other cases are similar. 
Recalling $L=a(y)\pt_y+b(y)\pt_x^2$, we need to estimate \[
\lw\|\nu\lambda  \int_0^{\tau } e^{\nu(\tau -s)B}\lw(b(y)\pt_x^2w(s)\rw)ds\rw\|_{\L^1_\rho}+\lw\|\nu\lambda  \int_0^{\tau } e^{\nu(\tau -s)B}a(y)\pt_y w(s)ds\rw\|_{\L^1_\rho}.\\
\]
Writing the above in Fourier and using the Green kernel decomposition \eqref{Stokes-green}, we get
\[\begin{cases}
\lw(\nu e^{\nu(\tau -s)B}\lw(b(y)\pt_x^2 w
\rw)\rw)_\al=&-\al^2 \nu \int_0^\infty H_\al(\tau -s,y,z)b(y)w_\al(s,y)dy\\
&-\al^2\nu \int_0^\infty R_\al(\tau -s,y,z)b(y)w_\al(s,y)dy,\\
\lw(\nu e^{\nu(\tau -s)B}\lw(a(y)\pt_y w\rw)\rw)_\al=&\nu\int_0^\infty H_\al(\tau -s,y,z)a(y)\pt_y w_\al(s,y)dy\\
&+\nu\int_0^\infty R_\al(\tau -s,y,z)a(y)\pt_y w_\al(s,y)dy.\end{cases}
\]

~\\
\textbf{Treating $\al^2 \nu \int_0^\infty H_\al(\tau -s,y,z)b(y)w_\al(s,y)dy$.}~
In view of $H_\al(\tau -s,y,z)$, we need to bound
\[
\lw\|\al^2\nu \int_0^{\tau } e^{-\al^2\nu(\tau -s)}\int_0^\infty \frac{1}{\sqrt{\nu(\tau -s)}}e^{-\frac{|y-z|^2}{4\nu(\tau -s)}}b(y)w_\al(s,y)dyds\rw\|_{\L^1_\rho}
\]
Here we recall that the $L^1_\rho$ norm is taken in $z$ near the boundary, when $0\le \Re z\le \delta_0+\rho$.
To gain analyticity near the boundary, we use 
\begin{equation}\label{exp-abc}
\bega 
e^{\eps_0(\delta_0+\mu-\Re z)|\al|}&\le e^{\eps_0(\delta_0+\mu-\Re y)_+|\al|} \cdot e^{\eps_0|\al||y-z|}\\
&\le  e^{\eps_0(\delta_0+\mu-\Re y)_+|\al|} e^{\eps_0\al^2 \nu(\tau -s)}\cdot e^{\eps_0\frac{|y-z|^2}{4\nu (\tau -s)}},
\enda 
\end{equation}
where the last two factors can be treated using $e^{-\al^2\nu(\tau -s)}e^{-\frac{|y-z|^2}{4\nu(\tau -s)}}$ in the heat kernel. Using this and the fact that the heat kernel is integrable in $z$, we obtain
\beq \label{ineq-1}
\bega
&\lw\|\al^2\nu \int_0^{\tau  }e^{-\al^2\nu(\tau -s)}\int_0^\infty \frac{1}{\sqrt{\nu(\tau -s)}}e^{-\frac{|y-z|^2}{4\nu(\tau -s)}}b(y)w_\al(s,y)dyds\rw\|_{L^1_\rho}\\
&\lesssim \al^2\nu \int_0^{\tau }e^{-\al^2\nu(\tau -s)}\int_0^\infty e^{\eps_0(\delta_0+\rho-\Re y)_+|\al|}b(y)|w_\al(s,y)|dyds.\\
\enda 
\eeq
Now since $b(y)=y\frac{2+\lambda y}{(1+\lambda y)^2}\le 2y$, the above can be bounded by 
\[\bega
&\al^2\nu \int_0^{\tau }e^{-\al^2\nu(\tau -s)} ds\sup_{0\le s\le \tau }\lw(\|yw_\al(s)\|_{\L^1_\rho}+\|yw_\al(s)\|_{L^1(y\ge \delta_0+\rho)}\rw)\\
&\lesssim \sup_{0\le s\le \tau }\lw(\|w_\al\|_{\L^1_\rho}+\|y^2w_\al\|_{L^2(y\ge \delta_0+\rho)}\rw).
\enda 
\]

~\\
\textbf{Treating $\al^2 \nu \int_0^\infty R_\al(\tau -s,y,z)b(y)w_\al(s,y)dy$.}~Using the bounds of the kernel $R_\al$ in Proposition \ref{Stokes-green}, we have 
\[\bega 
&e^{\eps_0(\delta_0+\rho-z)|\al|}\al^2 \nu \int_0^\infty R_\al(\tau -s,y,z)b(y)w_\al(s,y)dy\\
&\lesssim \al^2\nu \cdot e^{\eps_0(\delta_0+\rho-z)|\al|} \int_0^\infty \mu_fe^{-\theta_0\mu_f(y+z)}y|w_\al(s,y)|dy\\
&\lesssim \al^2\nu  \int_0^\infty \mu_fe^{-\frac{1}{2}\theta_0\mu_f(y+z)}e^{\eps_0(\delta_0+\rho-y)_+|\al|}y|w_\al(s,y)|dy\\
&=|\al| \nu \cdot(\mu_fe^{-\frac{\theta_0}{2}\mu_f z}) \int_0^\infty (|\al|ye^{-\frac{\theta_0}{4}\mu_f y})\cdot e^{-\frac{\theta_0}{4}\mu_fy}\cdot e^{\eps_0(\delta_0+\rho-y)_+|\al|}|w_\al(s,y)|dy\\
&\lesssim |\al|\nu \cdot \mu_fe^{-\frac{\theta_0}{2}\mu_fz}\int_0^\infty e^{-\frac{\theta_0}{4}\mu_fy}\cdot e^{\eps_0(\delta_0+\rho-y)_+|\al|}|w_\al(s,y)|dy\\
&=(|\al|\nu)(\mu_fe^{-\frac{\theta_0}{2}\mu_f z})\lw(\int_0^{\delta_0+\rho}+\int_{\delta_0+\rho}^\infty\rw)e^{-\frac{\theta_0}{4}\mu_fy}\cdot e^{\eps_0(\delta_0+\rho-y)_+|\al|}|w_\al(s,y)|dy\\
&\lesssim |\al|\nu\lw( \mu_fe^{-\frac{\theta_0}{2}\mu_fz}\rw)\|w_\al\|_{\L^1_\rho}+|\al|\nu \lw(\mu_fe^{-\theta_0\mu_f z}\rw)e^{-\frac{\theta_0\delta_0}{4}(|\al|+\nu^{-1/2})}\|yw_\al\|_{L^2(y\ge \delta_0+\rho)}.
\enda 
\]
Hence 
\[
\lw\|\al^2 \nu \int_0^\infty R_\al(\tau -s,y,z)b(y)w_\al(s,y)dy\rw\|_{\L^1_\rho}\lesssim  \nu \|\al w_\al\|_{\L^1_\rho}+\nu^2\|yw_\al\|_{L^2(y\ge \delta_0+\rho)}.
\]

~\\
\textbf{Treating $\nu \int_0^\infty H_\al(\tau -s,y,z)a(y)\pt_y w_\al(s,y)dy$.}~Integrating by parts, we have 
\[\bega 
&\quad \nu \int_0^\infty H_\al(\tau -s,y,z)a(y)\pt_y w_\al(s,y)dy\\
&=-\nu\int_0^\infty \pt_y (H_\al(\tau -s,y,z)a(y))w_\al(s,y)dy-\nu H_\al(\tau -s,0,z)a(0)w_\al(s,0).
\enda 
\]
Using the fact that $|\pt_y H_\al|\lesssim \frac{1}{{\nu(\tau -s)}}e^{-\theta_0\al^2\nu(\tau -s)}e^{-\theta_0\frac{|y-z|^2}{4\nu(t-s)}}$ and $|a'(y)|\lesssim 1$, we have 
\[\bega 
&\lw\|\nu \int_0^\infty \pt_y (H_\al(\tau -s,y,z)a(y))w_\al(s,y)dy\rw\|_{\L^1_\rho}\\
&\lesssim \lw(\frac{\sqrt{\nu}}{\sqrt{\tau -s}}+1\rw)\lw(\|w_\al(s)\|_{\L^1_\rho}+\|yw_\al(s)\|_{L^2(y\ge \delta_0+\rho)}\rw).
\enda
\]
For the boundary term $\nu H_\al(\tau -s,0,z)a(0)w_\al(s,0)$, we have 
\[
\lw\|\nu H_\al(\tau -s,0,z)a(0)w_\al(s,0)\rw\|_{\L^1_\rho}\lesssim \nu |w_\al(s,0)|e^{\eps_0|\al|(\delta_0+\rho)}.
\]
Hence 
\[\bega 
&\lw\|
\nu\int_0^\infty H_\al(\tau -s,y,z)a(y)\pt_yw_\al(y,s)dy
\rw\|_{\L^1_\rho}\\
&\lesssim \lw(\sqrt{\nu}(\tau -s)^{-1/2}+1\rw)\lw(\|w_\al(s)\|_{\L^1_\rho}+\|yw_\al(s)\|_{L^2(y\ge \delta_0+\rho)}\rw)+\nu |\w_\al(s,0)|e^{\eps_0|\al|(\delta_0+\rho)}.
\enda 
\]
Integrating both sides in time $s\in [0,\tau]$, we obtain 
\[
\bega 
&\lw\|\nu\int_0^\infty H_\al(\tau -s,y,z)a(y)\pt_yw_\al(y,s)dy\rw\|_{\L^1_\rho}\\
&\lesssim \lw(\sqrt{\nu \tau}+\tau \rw)\sup_{0\le s\le \tau}\lw(\|w_\al(s)\|_{\L^1_\rho}+\|y^2w_\al(s)\|_{L^2(y\ge \delta_0+\rho)}\rw)+\int_0^\tau \nu |\w_\al(s,0)|e^{\eps_0|\al|(\delta_0+\rho)} ds.
\enda
\]

~\\
\textbf{Treating $\nu\int_0^\infty R_\al(\tau -s,y,z)a(y)\pt_y w_\al(s,y)dy$.}~Integrating by parts, we get 
\[\bega 
&\nu\int_0^\infty R_\al(\tau -s,y,z)a(y)\pt_y w_\al(s,y)dy\\
&=-\nu\int_0^\infty \pt_y (R_\al(\tau -s,y,z)a(y))w_\al(s,y)dy -\nu R_\al(\tau -s,0,z)a(0)w_\al(s,0).
\enda 
\]
Since $|\pt_yR_\al|\lesssim \mu_f^{2}e^{-\theta_0\mu_f(y+z)}$, we have 
\[\bega
\lw\|\nu\int_0^\infty \pt_y (R_\al(\tau -s,y,z)a(y))w_\al(s,y)dy\rw\|_{\L^1_\rho}&\lesssim \nu \mu_f \| w_\al\|_{\L^1_\rho}+\nu^2 \|yw_\al\|_{L^2(y\ge \delta_0+\rho)}\\
&\lesssim \sqrt\nu \|w_\al\|_{\L^1_\rho}+\nu \|\al w_\al\|_{\L^1_\rho}+\nu^2 \|yw_\al\|_{L^2(y\ge \delta_0+\rho)}.
\enda 
\]
At the same time, we have 
\[
\lw\|\nu R_\al(\tau -s,0,z)a(0)w_\al(s,0)\rw\|_{\L^1_\rho}\lesssim \nu| w_\al(s,0)|e^{\eps_0(\delta_0+\rho)|\al|}
.
\]
The proof is complete.\end{proof}
In the next propostion, we estimate the boundary term appearing in \eqref{duh}.
\begin{proposition}\label{h-est}
There holds 
\[
\lw\|\int_0^\tau \Gamma(\nu(\tau-s))h(s)ds\rw\|_{\mathcal W_\rho^{k,1}}\lesssim \nu \lambda \int_0^\tau \lw\|w(s)|_{\{z=0\}}\rw\|_{\mathcal H_\rho^{k}}ds.
\]
\end{proposition}
\begin{proof}
By the estimate \eqref{semigroup}, we have 
\[
\lw\|\int_0^{\tau }\Gamma(\nu(\tau -s))h(s)ds
\rw\|_{\mathcal W^{k,1}_\rho}\lesssim \sum_{\al}|\al|^k\int_0^{\tau }|h_\al(s)|e^{\eps_0(\delta_0+\rho)|\al|}ds.
\]
From the identity \eqref{h-al}, we have $h=h_1+h_2$ where
\[
\begin{cases}
h_1&=-\lambda \nu\int_0^\infty e^{-|\al|y}w_\al(0,s)L_\al(e^{-|\al|y})dy, \\
h_2&=-\lambda \nu \int_0^\infty  e^{-|\al|y}L_\al \wtd w_\al^\star(y,s)dy.
\end{cases}
\]

~\\
\textbf{Treating $h_1$}.~Since $L_\al=a(y)\pt_y-\al^2 b(y)$, by a direct calculation, we have 
\[\bega 
h_1&=\lambda \nu w_\al(0,s)\lw(\int _0^\infty |\al|e^{-2|\al|y}a(y)dy+\al^2\int_0^\infty e^{-2|\al|y}b(y)dy\rw)\\
&\lesssim \lambda \nu|w_\al(0,s)|.
\enda 
\]
Here we use the fact that $a(y)\le 1, b(y)\le 2y$ and $\int |\al|e^{-|\al|y}dy\lesssim 1$.~

~\\
\textbf{Treating $h_2$}.~We have $h_2=h_{2,1}+h_{2,2}$ where 
\[
\begin{cases}
h_{2,1}&=-\lambda \nu \int_0^\infty e^{-|\al|y}a(y)\pt_y\wtd w_\al^\star(y,s)dy,\\
h_{2,2}&=\lambda \nu \al^2 \int_0^\infty e^{-|\al |y} b(y)\wtd w_\al^\star(s,y)dy.
\end{cases}
\]
We have 
\[\bega 
|h_{2,1}|e^{\eps_0(\delta_0+\rho)|\al|}&\lesssim \nu \lambda e^{\eps_0(\delta_0+\rho)|\al|}\int_0^\infty e^{-|\al|y}a(y)\pt_y\wtd w_\al^\star(y)dy\\
&\lesssim \nu \lambda \int_0^{\delta_0+\rho} e^{\eps_0(\delta_0+\rho-y)|\al|}|\pt_yw_\al^\star(y)|dy\\
&\quad+\nu\lambda  e^{\eps_0(\delta_0+\rho)|\al|} \int_{\delta_0+\rho}^\infty e^{-\frac{1}{2}|\al|(\delta_0+\rho)}a(y)|\pt_y\wtd w_\al^\star(y)|e^{-\frac{1}{2}|\al|y}dy\\
&\lesssim \nu\lambda \|\pt_y\wtd w_\al^\star(s)\|_{\L^1_\rho}+\nu \lambda e^{-(1/2-\eps_0)|\al|(\delta_0+\rho)}\|a(y)\pt_y\wtd w_\al^\star(s)\|_{L^\infty(y\ge\delta_0+\rho)}\frac{1}{|\al|}.
\\
\enda \]
Using the fact that $\lambda \lesssim |\al|$, we obtain 
\beq\label{est2}
|h_{2,1}|e^{\eps_0(\delta_0+\rho)|\al|}\lesssim \nu \lambda \|\pt_y\wtd w^\star_\al\|_{\L^1_\rho}+\nu e^{-(1/2-\eps_0)|\al|\delta_0}\|a(y)\pt_y\wtd w_\al^\star(s)\|_{L^\infty}.
\eeq
Now we recall from \eqref{w-star-eq} that $\wtd w_\al^\star$ solves the elliptic problem \[\bega
(\pt_y^2-\al^2+\lambda L_\al)\wtd w_\al^\star &=-\lambda w_\al(0)L_\al(e^{-|\al|y})=-\lambda w_\al(0)e^{-|\al|y}\lw(-|\al|a(y)-\al^2b(y)\rw)\\
&=\lambda w_\al(0)|\al|a(y)e^{-|\al|y}+\al^2\lambda w_\al(0)b(y)e^{-|\al|y}\\
&=\lambda w_\al(0)|\al|\lw(a(y)e^{-|\al|y}+|\al|e^{-|\al|y}b(y)\rw)
\enda 
\]
with the boundary condition $\wtd w_\al^\star|_{y=0}=0$.
Hence by using Lemma \ref{lem-ellip1}, we get
\[\bega
\|\pt_y\wtd w_\al^\star\|_{\L^1_\rho}&\lesssim \|\pt_y\wtd w_\al^\star\|_{\L^\infty_\rho}\\
&\lesssim \lambda |w_\al(0)||\al|\cdot\lw(\lw\|a(y)e^{-|\al|y}\rw\|_{\L^1_\rho}+\lw\|\al e^{-|\al|y}b(y)\rw\|_{\L^1_\rho}\rw)\\
&\quad + \lambda |w_\al(0)||\al|\lw(\|ya(y)e^{-|\al|y}\|_{L^2(y\ge \delta_0+\rho)}+\|\al e^{-|\al|y}b(y)\|_{L^2(y\ge \delta_0+\rho)}
\rw)\\
&\lesssim e^{\eps_0(\delta_0+\rho)|\al|}\cdot  \lambda |w_\al(0,s)|.
\enda
\]
Similarly, for the second term appearing on the right hand side of \eqref{est2}, we use Lemma \ref{lem-ellip2} to get 
\[\bega 
\|a(y)\pt_y\wtd w_\al^\star\|_{L^\infty}&\lesssim \int_0^\infty \lambda |w_\al(0)||\al| \cdot\lw| a(y)e^{-|\al|y}+|\al|e^{-|\al|y}b(y)\rw|dy\\
&\lesssim \lambda  |w_\al(0,s)|.
\enda 
\]
The bound for $h_{2,2}$ is nearly the same, as we note that $b(y)=\frac{y(2+\lambda y)}{(1+\lambda y)^2}\lesssim y a(y)$. We skip the details for $h_{2,2}$, and conclude that 
\[
|h_{2,2}|e^{\eps_0(\delta_0+\rho)|\al|}\lesssim \nu  |w_\al(0,s)|.
\]
Hence we get 
\[
|h_2|e^{\eps_0(\delta_0+\rho)|\al|}\lesssim \nu |w_\al(0,s)| e^{\eps_0(\delta_0+\rho)|\al|},
\]
giving the proposition. 
\end{proof}

Combining the previous two propositions, we have obtained the following. 

\begin{proposition}\label{coupled}Let $w$ be the solution to the Stokes problem \eqref{st-ext1} with the initial data $w_0$. Then for $k\ge 0$ and $\rho>0$, there hold the coupled semigroup estimates
\[\bega
\sup_{0\le s\le \tau}\|w(s)\|_{\mathcal W_\rho^{k,1}}&\lesssim \|w_0\|_{\mathcal W_\rho^{k,1}}+\|yD^{k+1}_{x,y}w_0\|_{L^2(y\ge \delta_0/2)} 
\\&\quad +\lambda \sup_{0\le s\le \tau} \Big( \|w(s)\|_{\mathcal W^{k,1}_\rho} + \|y^2D^{k+1}_{x,y}w(s)\|_{L^2(y\ge \delta_0+\rho)} \Big)
\\
&\quad+\lambda \nu \int_0^\tau \|w(s)\|_{\mathcal W_\rho^{k+1,1}}+\lambda \nu \int_0^\tau \|w(s)_{z=0}\|_{\mathcal H^k_\rho}ds\\
&\quad +\int_0^\tau\lw(\|f(s)\|_{\mathcal W_\rho^{k,1}}+\|yD^{k+1}_{x,y} f(s)\|_{L^2(y\ge \delta_0+\rho)}\rw)ds+\int_0^\tau\|g(s)\|_{\mathcal H_\rho^{k}}ds .
\enda 
\]
\end{proposition}

\begin{proof} Recall the Duhamel representation \eqref{duh}. The proof thus follows directly by combining the semigroup estimates \eqref{semigroup}, the estimates for the perturbation term in Proposition \ref{nuL}, and the boundary estimates in Proposition \ref{h-est}.
\end{proof}

Finally, we give bounds on $\|w(s)_{z=0}\|_{\mathcal H^k_\rho}$ appearing on the right of the previous estimates.

\begin{proposition}\label{at-bdr}
Let $w$ be the solution to the Stokes problem \eqref{st-ext1}. There holds
\beq\label{l1-time}\bega
\lw\|w(\tau)|_{z=0}\rw\|_{\mathcal H_\rho^{k}}&\lesssim \|y\pt_xw_0\|_{L^2(y\ge\delta_0/2)}+(\nu \tau)^{-1/2}\|w_0\|_{\mathcal W^{k+1,1}_\rho}+\nu \lambda  \int_0^\tau \lw\|w(s)|_{z=0}\rw\|_{\mathcal H^{k+1}_\rho}ds\\
&\quad+\lambda \sup_{0\le s\le\tau}\lw(\|w(s)\|_{\mathcal W_\rho^{k,1}}+\|y^2D^{k+1}_{x,y}w(s)\|_{L^2(y\ge \delta_0+\rho)}\rw)+\nu \int_0^\tau \|\pt_xw(s)\|_{\mathcal W^{k+1,1}_\rho}ds\\
&\quad+\int_0^\tau \|yD_x^1f(s)\|_{L^2(y\ge \delta_0+\rho)}+\nu^{-1/2}\int_0^\tau (\tau-s)^{-1/2}\|f(s)\|_{\mathcal W_\rho^{k,1}}ds
\\
&\quad+\int_0^\tau \|\pt_x f (s)\|_{\mathcal W_\rho^{k,1}}ds+\int_0^\tau \|g(s)\|_{\mathcal H_\rho^{k}}ds.\enda
\eeq
\end{proposition}

\begin{proof}
We shall bound each term in \eqref{duh}, evaluating at $z=0$. First, we have 
\[\bega
w_\al(\tau ,0)=&\int_0^\infty G_\al(\tau ,y,0)w_{0,\al}(y)dy+\int_0^\tau G_\al(\tau-s,y,0)(\nu\lambda L_\al w_\al)(s,y)dy\\
&+\int_0^{\tau }\int_0^\infty G_\al(\tau -s,y,0)f_\al(s,y)dyds+\int_0^{\tau }\Gamma_\al(\nu(\tau -s))(g_\al+h_\al)(s)ds\\
=&P_1(\tau)+P_2(\tau)+P_3(\tau)+P_4(\tau).
\enda
\]
Hence we get 
\[
|\al|^ke^{\eps_0(\delta_0+\rho)|\al|}w_\al(\tau,0)=\sum_{i=1}^4 P_i(\tau)
\]
where 
\[
\begin{cases}
P_1(\tau)=&|\al|^ke^{\eps_0(\delta_0+\rho)|\al|}\int_0^\infty G_\al(\tau ,y,0)w_{0,\al}(y)dy\\
P_2(\tau)=&|\al|^k e^{\eps_0(\delta_0+\rho)|\al|}\int_0^\tau G_\al(\tau-s,y,0)(\nu\lambda L_\al w_\al)(s,y)dy\\
P_3(\tau)=&|\al|^ke^{\eps_0(\delta_0+\rho)|\al|}\int_0^{\tau }\int_0^\infty G_\al(\tau -s,y,0)f_\al(s,y)dyds\\
P_4(\tau)=&|\al|^ke^{\eps_0(\delta_0+\rho)|\al|}\int_0^{\tau }\Gamma_\al(\nu(\tau -s))(g_\al+h_\al)(s)ds.
\end{cases}
\]
We recall the pointwise Green kernel bound:
\[
G_\al(\tau -s,y,0)\lesssim (\nu(\tau -s))^{-1/2}e^{-\theta_0\frac{y^2}{\nu(\tau -s)}}e^{-\theta_0\al^2\nu(\tau -s)}+\mu_fe^{-\mu_f y}.
\]
Let us first bound the term 
\beq\label{P3}
P_3(\tau)=|\al|^k e^{\eps_0(\delta_0+\rho)|\al|}
\lw| 
\int_0^\tau  \int_0^\infty G_\al(\tau-s,y,0)f_\al(s,y)dyds\rw|
\eeq
We will show that 
\beq\label{P3ineq}
P_3(\tau)\lesssim \int_0^\tau \lw(\|yf_\al(s)\|_{L^2(y\ge \delta_0+\rho)}+ \lw((\nu(\tau-s))^{-1/2}+|\al|+\nu^{-1/2}\rw)|\al|^k \|f_\al(s)\|_{\L^1_\rho}\rw)ds .
\eeq
To show the above inequality, we split the integral in $y$ in \eqref{P3} into $\int_{\delta_0+\rho}^\infty+\int_0^{\delta_0+\rho}$. 
We note that if $y\ge \delta_0+\rho$, then $G_\al$ is exponentially decay in $\al$, which is faster than $e^{-\eps_0(\delta_0+\rho)|\al|}$ for $\eps_0$ small, giving 
\[\bega
&|\al|^k e^{\eps_0(\delta_0+\rho)|\al|}
\lw| 
\int_0^\tau  \int_{\delta_0+\rho}^\infty G_\al(\tau-s,y,0)f_\al(s,y)dyds
\rw|\\
&\lesssim \int_0^\tau \|yf_\al(s)\|_{L^2(y\ge \delta_0+\rho)}ds
\enda 
\]
Now we consider $y\le \delta_0+\rho$. By the Cauchy inequality $\al^2\nu(\tau-s)+\frac{y^2}{\nu(\tau-s)}\ge 4|\al|y$ and the fact that $\theta_0,\eps_0$ is taken to be small, we obtain 
\[\bega
e^{\eps_0(\delta_0+\rho)|\al|}G_\al(\tau-s,0,y)&\lesssim \lw((\nu(\tau-s)^{-1/2}+\mu_f\rw)e^{-\frac{\theta_0}{2}|\al|y}e^{\eps_0(\delta_0+\rho)|\al|}\\
&\lesssim \lw((\nu(\tau-s))^{-1/2}+\mu_f\rw)e^{\eps_0(\delta_0+\rho-y)|\al|}.
\enda \]
Hence we obtain
\[\bega
&|\al|^k e^{\eps_0(\delta_0+\rho)|\al|}\lw|\int_0^\tau \int_0^{\delta_0+\rho}G_\al(\tau-s,y,0)f_\al(s,y)dyds\rw|\\
&\lesssim  \int_0^\tau \lw((\nu(\tau-s))^{-1/2}+|\al|+\nu^{-1/2}\rw)|\al|^k \|f_\al(s)\|_{\L^1_\rho}ds.
\enda
\]
This concludes the proof for the inequality \eqref{P3ineq}.
Next, we bound 
\[
P_2(\tau)=|\al|^k e^{\eps_0(\delta_0+\rho)|\al|}\int_0^\tau G_\al(\tau-s,y,0)(\nu\lambda L_\al w_\al)(s,y)dy.
\]
The proof for the bound of $P_2(\tau)$ is exactly the same as in the semigroup estimate in Proposition \ref{nuL}, except now that we cannot use the $L^1$ norm in $z$ in this case, as $z=0$, giving an extra $\mu_f=|\al|+\nu^{-1/2}$ in the estimate involving the kernel $R_\al(\tau-s,0,z)|_{z=0}$. 
We obtain 
\[\bega
P_2(\tau)&\lesssim \lambda \sup_{0\le s\le\tau}\lw(\|w\|_{\mathcal W_\rho^{k,1}}+\|y^2D^{k+1}_{x,y}w\|_{L^2(y\ge \delta_0+\rho)}\rw)\\
&+\nu \int_0^\tau  \|\pt_xw\|_{\mathcal W^{k+1,1}_\rho}+\nu \lambda \sum_{\al} \int_0^\tau |\al|^{k+1} |w_\al(s,0)|e^{\eps_0(\delta_0+\rho)|\al|}.
\enda
\]
Finally, for the initial data, we obtain 
\[
P_1(\tau)\lesssim \|y\pt_x^kw_0\|_{L^2(y\ge \delta_0/2)}+(\nu \tau)^{-1/2}\|w_0\|_{\mathcal W^{k+1,1}_\rho},
\]
giving the proposition. 
 \end{proof}

\begin{remark} Note that in the above estimates, the boundary value quantity $\|w(\tau)_{z=0}\|_{\mathcal H_\rho^k}$ has  two losses of derivatives compared to the norm $\|w(\tau)\|_{\mathcal W_\rho^{k,1}}$. However, it has only one loss of derivative compared to its norm and we are able to close the Sobolev-analytic estimates by introducing an iterative adjusted $k$-index norms, yielding close estimates on the Stokes semigroup in terms of initial and boundary data $f,g$ given in the problem \eqref{st-ext1}. 
\end{remark}

\begin{proof}[Proof of Theorem \ref{theo-Stokes}] Let $w = e^{\nu t S} w_0$ be the solution to \eqref{st-ext1} with $f=0$ and $g=0$. In view of the previous propositions, we define the following norm 
\beq\label{main-norm}
\bega 
\mathcal A_k(w(\tau),\rho)&=\lw(\|w(\tau )\|_{\mathcal W^{k,1}_\rho}+ \sqrt{\nu \tau}\|w(\tau)|_{z=0}\|_{\mathcal H^{k-1}_\rho}\rw)\\
&\quad+\lw(\|w(\tau )\|_{\mathcal W^{k+1,1}_\rho}+\sqrt{\nu \tau}\|w(\tau)|_{z=0}\|_{\mathcal H_\rho^{k}}\rw)(\rho_0-\rho-\beta \tau )^{\gamma}\\
\enda 
\eeq
and the quantity 
\[\bega
A(\beta)=&\sup_{0<\tau  \beta <\rho_0}\quad \lw\{\sup_{0<\rho<\rho_0-\beta \tau }\lw(\mathcal A_k(w(\tau),\rho) \rw)\rw\}+\sup_{0<\tau \beta<\rho_0}\|y^2D_{x,y}^5w\|_{L^2(y\ge \delta_0/2)}.
\enda 
\]
We claim that 
\begin{equation}\label{bound-Abeta}
A(\beta)\lesssim
\|w_0\|_{\mathcal W^{2,1}_{\rho_0}}+\|y^2D_{x,y}^5w_0\|_{L^2(y\ge \delta_0/4)} 
+e^{C(1+A(\beta))\beta^{-1}} \|y^2D_{x,y}^5w_0\|_{L^2(y\ge \delta_0/4)}
\end{equation}
which would yield the theorem. In fact, in Section \ref{sec7}, using precisely Propositions \ref{coupled} and \ref{at-bdr} above, we shall prove the claim for the nonlinear solution to \eqref{st-ext1} with $f$ and $g$ being the nonlinear terms inherited from the vorticity formulation of the Navier-Stokes problem. We therefore skip to repeat the details here for the linear problem with zero $f$ and $g$. 
\end{proof}

\section{Elliptic estimates}\label{ellip-sec}
In this section, we prove estimates for velocity near the boundary and away from the boundary in terms of vorticity. In particular, we consider the elliptic problem 
\beq\label{ellip-new}
\begin{cases}
&(\triangle+\lambda L)\psi=\lambda^2 w,\\
&\psi|_{y=0}=0.
\end{cases}
\eeq
where $L\psi=a(y)\pt_y\psi+b(y)\pt_x^2\psi$, $a(y)=\frac{1}{1+\lambda y},\quad b(y)=\frac{y(2+\lambda y)}{(1+\lambda y)^2}$.
Our main goal in this section is to show the elliptic estimates in the analytic domain near the boundary (see Proposition \ref{near-bdr-prop} below), in the intermediate region (Proposition \ref{inter2}) and
the region away from the boundary (Proposition \ref{D-k-new}).

\subsection{Elliptic estimates in the analytic region}
We first show the following lemma that gives a $L^\infty$ bound for velocity field:

\begin{lemma} \label{lem-ellip2}There holds 
\[
\|a(y)\al\psi_\al\|_{L^\infty}+\|\pt_y \psi_\al\|_{L^\infty}
\lesssim \int_0^\infty |w_\al(y)|dy 
\lesssim \|w_\al\|_{\L^1_\rho}+\|yw_\al\|_{L^2(y\ge \delta_0+\rho)}.
\]
\end{lemma}

\begin{proof}
We recall the original elliptic problem on the written in the variables $(\theta,r)\in \mathbb{T}\times [1,\infty)$:
\[
\lw(\pt_r^2+\frac{1}{r}\pt_r-\frac{n^2}{r^2}\rw)\psi_n=\w_n
\]
where $\psi(t,r,\theta)=\psi (\lambda^2t,1+\lambda y,\lambda x)$. We note that $\al=\lambda n$ where $n$ is the original frequency before making the change of variables.
The solution to the above elliptic problem in the original variables is given by 
\beq\label{psi-n}
-\psi_n(r)=\frac{1}{2|n|}\int_1^r \frac{s^{1+|n|}-s^{1-|n|}}{r^{|n|}}\w_n(s)ds+\frac{1}{2|n|}\int_r^\infty \lw(s^{1-|n|}r^{|n|}-\frac{s^{1-|n|}}{r^{|n|}}\rw)\w_n(s)ds.
\eeq
Since the function $s^{1+|n|}-s^{1-|n|}$ is increasing on $[1,r]$ and the function $s^{1-|n|}r^{|n|}-\frac{s^{1-|n|}}{r^{|n|}}$ is decreasing on $[r,\infty)$, we get the pointwise estimate
\[
|n\psi_n(r)|\lesssim (r-r^{1-2n})\int_1^\infty |\w_n(s)|ds\lesssim r\|\w_n\|_{L^1(1,\infty)}.
\]
Hence we obtain 
\beq \label{infty-norm}
\lw\|
\frac{n\psi_n(r)}{r}
\rw\|_{L^\infty}\lesssim \|\w_n\|_{L^1(1,\infty)}.
\eeq
Now in the rescaled variables $(\al,y)$, we get 
\[
 |\al|\|a(y)\psi_\al\|_{\L^\infty}\lesssim \int_0^\infty |w_\al(y)|dy\lesssim \|w_\al\|_{\L^1_\rho}+\|yw_\al\|_{L^2(y\ge \delta_0+\rho)},
\]
upon noting that $a(y)=\frac{1}{1+\lambda y}=\frac{1}{r}$. 
Now we show that $\|a(y)\pt_y\psi_\al\|_{L^\infty}\lesssim \int_0^\infty |w_\al(y)|dy$.
By a direct calculation, we get
\[
-2\psi'_n(r)=-r^{-|n|-1}\int_1^r (s^{1+|n|}-s^{1-|n|})\w_n(s)ds+\lw(r^{|n|-1}-r^{-|n|-1}\rw)\int_r^\infty s^{1-n}w_n(s)ds.
\]
Hence 
\[
|\psi_n'(r)|\lesssim \int_1^\infty |w_n(s)|ds.
\]
The proof is complete.
\end{proof}

\begin{remark} It is known from Section 2.2 of \cite{Mae-ext}, that the Biot-Savart law \ref{psi-n} defines a unique velocity that decays at infinity, under the decaying assumption $r^{1-|n|}\w_n\in L^1$. In our current work, the vorticity satisfies the decaying assumption $\|r^2D_{x,y}^3\w\|_{L^2_{r,\theta}}<\infty$,  hence the Biot-Savart law  \eqref{psi-n} gives a unique velocity solution for all $|n|\ge 1$. We also note that when $n=0$, the stream function equation reduces to 
\[
\pt_r^2\psi_0+\frac{1}{r}\pt_r\psi_0=\w_0
\]
giving $\psi'_0(r)=\frac{1}{r}\int_1^r s\w_0(s) ds$. This gives the Biot-Savart law $(u_r,u_\theta)=\lw(0,-\frac{1}{r}\int_1^r s\w_0(s) ds\rw)$ for $n=0$. We also note that when the frequency $n=\al=0$, the analytic norm in $x$ (or $\theta$) reduces to Sobolev norm.
\end{remark}

In the next lemma, we derive the elliptic estimate for velocity in the analytic norm near the boundary:
\begin{lemma} \label{lem-ellip1}For $\lambda,\delta_0$ and $\rho$ small, there holds 
\[\bega
\|\nabla\psi_\al\|_{\L^\infty_\rho}&\lesssim \|w_\al\|_{\L^1_\rho}+\|yw_\al\|_{L^2(y\ge \delta_0+\rho)},\\
\enda 
\]
\end{lemma}
\begin{proof} 
We first show that 
\[
\|\nabla\psi_\al\|_{\L^\infty_\rho}\lesssim \|w_\al\|_{\L^1_\rho}+\|yw_\al\|_{L^2(y\ge \delta_0+\rho)}
\]
Since $\psi_\al$ solves \[
(\pt_y^2-\al^2)\psi_\al=\lambda^2w_\al-\lambda a(y)\pt_y\psi_\al+\lambda \al^2 b(y)\psi_\al\]
with the boundary condition $\psi_\al|_{y=0}=0$, we get 
\[\bega
2\al \psi_\al(z)=&\lambda^2 \int_0^\infty \lw(e^{-\al(y+z)}-e^{-\al|y-z|}\rw)w_\al(y)dy-\lambda \int_0^\infty \lw(e^{-\al(y+z)}-e^{-\al|y-z|}\rw)a(y)\pt_y\psi_\al(y)dy\\
&+\lambda \int_0^\infty \lw(e^{-\al(y+z)}-e^{-\al|y-z|}\rw)\al^2 b(y)\psi_\al(y)dy\\
&=I_1+I_2+I_3,
\enda 
\]
where 
\[
\begin{cases}
I_1&=\lambda^2 \int_0^\infty \lw(e^{-\al(y+z)}-e^{-\al|y-z|}\rw)w_\al(y)dy,\\
I_2&=-\lambda \int_0^\infty \lw(e^{-\al(y+z)}-e^{-\al|y-z|}\rw)a(y)\pt_y\psi_\al(y)dy,\\
I_3&=\lambda \int_0^\infty \lw(e^{-\al(y+z)}-e^{-\al|y-z|}\rw)\al^2 b(y)\psi_\al(y)dy.
\end{cases}
\]
\textbf{Treating $I_1$}. Using the first estimate in \eqref{exp-abc}, we simply bound 
\[\bega 
|I_1|e^{\eps_0(\delta_0+\rho-z)|\al|}&\lesssim \int_0^{\delta_0+\rho}e^{-\frac{|\al|}{2}|y-z|}e^{\eps_0(\delta_0+\rho-y)|\al|} |w_\al(y)|dy+\int_{\delta_0+\rho}^\infty |w_\al(y)|dy\\
&\lesssim  \|w_\al\|_{\L^1_\rho}+\|yw_\al\|_{L^2(y\ge \delta_0+\rho)}.\\
\enda 
\] 
\textbf{Treating $I_2$.} We will show that 
\[
I_2\lesssim \|\al\psi_\al\|_{\L^\infty_\rho}\ln(1+\lambda(\delta_0+\rho))+\lambda \|a(y)\psi_\al\|_{L^\infty(y\ge \delta_0+\rho)}
\]
Since 
\[
I_2=-\lambda  \int_0^\infty \lw(e^{-\al(y+z)}-e^{-\al|y-z|}\rw)a(y)\pt_y\psi_\al(y)dy,
\]
we use integration by parts to get 
\[\bega
I_2&\lesssim \lambda \int_0^\infty |\al|e^{-|\al||y-z|}a(y)|\psi_\al(y)|dy+\lambda \int_0^\infty |a'(y)|e^{-|\al||y-z|}|\psi_\al(y)|dy\\
&\lesssim  \lambda \int_0^\infty |\al|e^{-|\al|y-z|}a(y)|\psi_\al(y)|dy+\lambda^2 \int_0^\infty a(y)^2 e^{-|\al||y-z|}|\psi_\al(y)|dy\\
\enda 
\]
Therefore, 
\beq\label{re-use2}
\bega
|I_2|e^{\eps_0(\delta_0+\rho-z)|\al|}&\lesssim \lambda  \|\al\psi_\al\|_{\L^\infty_\rho}\int_0^{\delta_0+\rho}\frac{1}{1+\lambda y}dy+\lambda \|a(y)\psi_\al\|_{L^\infty(y\ge \delta_0+\rho)}\\
&\quad +\frac{\lambda^2}{|\al|}\|\al\psi_\al\|_{\L^\infty_\rho}\int_0^{\delta_0+\rho}\frac{1}{1+\lambda y}dy+\frac{\lambda^2}{|\al|}\|a(y) \psi_\al\|_{L^\infty(y\ge\delta_0+\rho)}\\
&\lesssim \|\al\psi_\al\|_{\L^\infty_\rho}\ln(1+\lambda(\delta_0+\rho))+\lambda \|a(y)\psi_\al\|_{L^\infty(y\ge \delta_0+\rho)}\\
&\lesssim \|\al\psi_\al\|_{\L^\infty_\rho}\ln(1+\lambda(\delta_0+\rho))+ \|a(y)\al \psi_\al\|_{L^\infty(y\ge \delta_0+\rho)}.
\enda 
\eeq
where we use the fact that $|\al|\ge \lambda$ whenever $\al\neq 0$.~\\
\textbf{Treating $I_3$}. We will show that 
\beq\label{re-use1}
\|I_3\|_{\L^\infty_\rho}\lesssim \lambda \|\nabla\psi_\al\|_{\L^\infty_\rho}+\lw\|a(y)\al \psi_\al \rw\|_{L^\infty(y\ge \delta_0+\rho)} .
\eeq
Indeed, if $y\le \delta_0+\rho$, then $b(y)\le 2y\le 2(\delta_0+\rho)$, and hence 
\[
e^{\eps_0(\delta_0+\rho-z)|\al|}\int_0^{\delta_0+\rho}|\al|e^{-|\al||y-z|}b(y)dy\le 2(\delta_0+\rho)\int_0^\infty |\al|e^{-\frac{1}{2}|\al||y-z|}dy\lesssim 1.
\]
If $y\ge \delta_0+\rho$, then we have 
\[\bega 
&e^{\eps_0(\delta_0+\rho-z)|\al|} \int_{\delta_0+\rho}^\infty e^{-|\al||y-z|}|\al| \frac{\lambda y(2+\lambda y)}{(1+\lambda y)^2}(\al\psi_\al(y))dy\\
&=\lambda e^{\eps_0(\delta_0+\rho-z)|\al|}\int_{\delta_0+\rho}^\infty \frac{|\al||\psi_\al(y)|}{1+\lambda y} \lw(e^{-|\al||y-z|}|\al|y\rw) \cdot\frac{2+\lambda y}{1+\lambda y}dy\\
 &\lesssim  \lw \| \frac{\al \psi_\al}{1+\lambda y}\rw\|_{L^\infty(y\ge \delta_0+\rho)}\lambda \int_{\delta_0+\rho}^\infty |\al||e^{-\frac{1}{2}|\al||y-z|}(|y-z|+z)dy\\
 &\lesssim \lw(\frac{\lambda}{|\al|}+\lambda (\delta_0+\rho)\rw) \lw\|\frac{\al \psi_\al}{1+\lambda y}\rw\|_{L^\infty(y\ge {\delta_0+\rho})} \lesssim  \lw\|a(y)\al \psi_\al\rw\|_{L^\infty(y\ge {\delta_0+\rho})} .\enda 
\]
since $\lambda \lesssim |\al|$.
In summary, we get 
\beq\label{est1}
\bega 
\|\nabla \psi_\al\|_{\L^\infty_\rho}&\le C_0\lw( \|w_\al\|_{\L^1_\rho}+\|yw_\al\|_{L^2(y\ge \delta_0+\rho)}+\lw\|a(y)\al \psi_\al \rw\|_{L^\infty(y\ge \delta_0+\rho)}\rw)\\
&\quad +C_0\lw( \|\nabla \psi_\al\|_{\L^\infty_\rho}\rw)\lw(\lambda+\ln(1+\lambda(\delta_0+\rho))\rw) .
\enda 
\eeq
We note that in the estimate above, the constant $C_0$ does not depend on $\al$ and $\lambda$.
Taking $\lambda$ to be small so that 
\[
\lambda+\ln(1+\lambda(\delta_0+\rho))\le \frac{1}{2C_0},
\]
the last term in the estimate \eqref{est1} can be absorbed to the left hand side, giving 
\[
\|\nabla \psi_\al\|_{\L^\infty_\rho}\lesssim \|w_\al\|_{\L^1_\rho}+\|yw_\al\|_{L^2(y\ge \delta_0+\rho)}+\lw\|a(y)\al \psi_\al \rw\|_{L^\infty(y\ge \delta_0+\rho)}.
\]
Finally, using Lemma \eqref{lem-ellip2} for the last time in the above, we obtain 
\[
\|\nabla \psi_\al\|_{\L^\infty_\rho}\lesssim \|w_\al\|_{\L^1_\rho}+\|yw_\al\|_{L^2(y\ge \delta_0+\rho)}.
\]
The proof is complete.
\end{proof}

In order to close the estimate for velocity in terms of vorticity, we need the following lemma

\begin{proposition} \label{near-bdr-prop} Let $\psi$ be the solution to the elliptic problem \eqref{ellip-new}, and set $\wtd u=\nabla^\perp\psi$. For $k\in\{0,1\}$, there hold
\[\bega 
\|\pt_x^k \wtd u\|_{\L^\infty_\rho}&\lesssim \|w\|_{\mathcal W^{k,1}_\rho}+\|yD^{k+1}_{x,y}w\|_{L^2(y\ge \delta_0+\rho)},\\
\|(y\pt_y)^k\wtd u\|_{\L^\infty_\rho}&\lesssim \|w\|_{\mathcal W^{k,1}_\rho}+\|yD^{k+1}_{x,y} w\|_{L^2(y\ge \delta_0+\rho)},\\
\|y^{-1}\pt_x\psi\|_{\mathcal W^{k,\infty}_\rho}&\lesssim \| w\|_{\mathcal W^{k,1}_\rho} + \|\partial_{ x}w\|_{\cW^{k,1}_\rho} + \|yD_{x,y}^{k+2}w\|_{L^2(y\ge \delta_0 + \rho)}.
\enda 
\]
\end{proposition}
\begin{proof}
First, when $k=0$, from Lemma \ref{lem-ellip1}, we have 
\[\bega 
\|\wtd u\|_{\L^\infty_\rho}&\lesssim \|w\|_{\L^1_\rho}+\sum_{\al\in \lambda\Z}\|yw_\al\|_{L^2(y\ge \delta_0+\rho)}\\
&\lesssim \|w\|_{\L^1_\rho}+\|yD_xw\|_{L^2(y\ge \delta_0+\rho)}.\\
\enda 
\]
Now we give the proof for $k=1$. Since $\pt_x\psi$ solves the same elliptic problem with the condition $\pt_x\psi|_{y=0}=0$, we obtain 
\[
\|\pt_x \wtd u\|_{\L^\infty_\rho}\lesssim \|\pt_x w\|_{\L^1_\rho}+\|yD_{x,y}^2w\|_{L^2(y\ge \delta_0+\rho)}.
\]
Now for $\|y\pt_y\wtd u\|_{\L^\infty_\rho}$, we note that 
\[
\bega 
\|y\pt_y(\pt_x\psi)\|_{\L^\infty_\rho}\lesssim \|\pt_y(\pt_x\psi)\|_{\L^\infty_\rho}\lesssim \|\pt_xw\|_{\L^1_\rho}+\|yD_{x,y}^2w\|_{L^2(y\ge \delta_0+\rho)}.
\enda 
\]
Now we have 
\[
y\pt_y(\pt_y\psi)=y\pt_y^2\psi=y(\lambda^2w-\lambda a(y)\pt_y\psi-\lambda b(y)\pt_x^2\psi-\pt_x^2\psi).
\]
Hence we get
\[\bega
\|y\pt_y^2\psi\|_{\L^\infty_\rho}&\lesssim \|yw\|_{\L^\infty_\rho}+\|\wtd u\|_{\L^\infty_\rho}+\|\pt_x\wtd u\|_{\L^\infty_\rho}\\
&\lesssim \lw( \|w\|_{\L^1_\rho}+\|y\pt_yw\|_{\L^1_\rho}\rw)+\|\pt_xw\|_{\L^1_\rho}+\|yD^2_{x,y}w\|_{L^2(y\ge \delta_0+\rho)}\\
&\lesssim \|w\|_{\mathcal W^{1,1}_\rho}+\|yD^2_{x,y}w\|_{L^2(y\ge \delta_0+\rho)}.
\enda 
\]
The proof is complete.
Now we show the last inequality stated in this proposition.
When $k=0$, we have, for any $y\le\delta_0+\rho$:
\[\bega
\pt_x\psi&=\int_0^y \pt_z(\pt_x\psi)(z)dz.\enda 
\]
Since $e^{\eps_0|\al|(\delta_0+\rho-y)}\le e^{\eps_0|\al|(\delta_0+\rho-z)}$, we have
\[
\|y^{-1}\pt_x\psi\|_{\L^\infty_\rho}\lesssim \|\pt_x\wtd u\|_{\L^1_\rho}\lesssim \|\pt_x\wtd u\|_{\L^\infty_\rho}\lesssim \|\pt_xw\|_{\L^1_\rho}+\|yD_{x,y}^2w\|_{L^2(y\ge \delta_0+\rho)}.
\]
For $k=1$, we note that 
\[\begin{cases}
&\pt_x(y^{-1}\pt_x\psi)=y^{-1}\pt_x(\pt_x\psi),\\
&y\pt_y(y^{-1}\pt_x\psi)=-\frac{1}{y}\pt_x\psi+y\pt_y\pt_x\psi.
\end{cases}
\]
Hence 
\[\bega 
\|y\pt_y(y^{-1}\pt_x\psi)\|_{\L^\infty_\rho}&\le  \|y^{-1}\pt_x(\pt_x\psi)\|_{\L^\infty_\rho}+\|\pt_y(\pt_x\psi)\|_{\L^\infty_\rho}\\
&\lesssim \|\pt_x^2w\|_{\L^1_\rho}+ \|yD^3_{x,y}w\|_{L^2(y\ge \delta_0+\rho)}.\enda 
\]
The proposition follows. 
\end{proof}

\subsection{Elliptic estimates in the intermediate region}\label{inter-sec}
We also need the following elliptic estimates in the intermediate region away from the boundary. We first prove the following elementary lemma. 

\begin{lemma} \label{inter-nice}Assume $0<\delta_1<\delta_2<\delta_0$ and let $c\in (0,1)$ be any constant such that $\delta_2<c\delta_0$. Then for any function $F_\al(y)$ and $k\ge 1$, there holds 
\[
|\al|^k\int_0^\infty e^{-|\al||y-z|}|F_\al(y)|dy\le C \lw(\|F_\al\|_{\L^1_\rho}+\int_{c\delta_0}^\infty e^{-\frac{1}{2}|\al||y-z|}|F_\al(y)|dy\rw)
\]
for $z\in [\delta_1,\delta_2]$. The constant $C$ depends only on $\delta_1,\delta_2,\delta_0$ and $k$.
\end{lemma}

\begin{proof}
Splitting the integral in $y$ into $y\le c\delta_0$ and $y\ge c\delta_0$, we have two cases:

~\\
\textbf{Case 1.} $y\le c\delta_0$.~In this case, we get $y\le \delta_0+\rho$, and moreover 
\[
e^{-\eps_0|\al|(\delta_0+\rho-y)}\le e^{-\eps_0(1-c)|\al|\delta_0} .
\]
Hence 
\[\bega
\int_{0}^{c\delta_0}|\al|^ke^{-|\al||y-z|}|F_\al(y)|dy&\le \int_0^{c\delta_0}e^{-|\al||y-z|} |\al|^k e^{-\eps_0(1-c)|\al|\delta_0}e^{\eps_0|\al|(\delta_0+\rho-y)}|F_\al(y)|dy\\
&\lesssim \|F_\al\|_{\L^1_\rho} .
\enda 
\]

~\\
\textbf{Case 2.} $y\ge c\delta_0$.~In this case we have $|y-z|\ge c\delta_0-\delta_2$. And hence $e^{-\frac{1}{2}|\al||y-z|}\le e^{-\frac{1}{2}(c\delta_0-\delta_2)|\al|}$. 
\[\bega 
&\int_{c\delta_0}^{\infty}|\al|^ke^{-|\al||y-z|}|F_\al(y)|dy\le \int_{c\delta_0}^\infty  |\al|^k e^{-\frac{1}{2}(c\delta_0-\delta_2)|\al|}e^{-\frac{1}{2}|\al|y-z|}|F_\al(y)|dy\\
&\lesssim \int_0^\infty e^{-\frac{1}{2}|\al||y-z|}|F_\al(y)|dy .
\enda 
\]
The proof is complete.
\end{proof}

\begin{proposition} \label{inter2}Let $\psi$ be the solution to the elliptic problem \eqref{ellip-new}, and set $\wtd u=\nabla^\perp\psi$. Then for any $\delta_1<\delta_2<\delta_0$, we have
\[\bega 
\|D^k_{x,y}\wtd u\|_{L^\infty(\delta_1\le y\le \delta_2)}&\lesssim \|w\|_{\L^1_\rho}+\|yD_xw\|_{L^2(y\ge c\delta_0)}  \\
\enda 
\] 
where $c\in (0,1)$ is any constant such that $c\delta_0\in (\delta_2,\delta_0)$.
\end{proposition}

\begin{proof}
We give the proof for $\pt_x^3\wtd u$ and $\pt_y^3\wtd u$ only. The other cases are similar.
Since $\psi$ solves
\[
\triangle \psi=\lambda^2w-\lambda L\psi,\qquad \psi|_{y=0}=0,
\]
we use the Green kernel for the Laplacian $(\pt_y^2-\al^2)$ and integrating by parts for the term $a(y)\pt_y\psi$, to get
\[
|\al|^3 |\wtd u_\al(z)|\lesssim |\al|^3 \int_0^\infty e^{-|\al||y-z|}\lw(\lambda^2|w_\al(y)|+\lambda |\al| a(y)|\psi_\al(y)|+\lambda |a'(y)||\psi_\al(y)|+\lambda \al^2b(y)|\psi_\al(y)|
\rw)dy .
\]
Applying Lemma \ref{inter-nice} for three terms on the right hand side in the above, we get
\beq\label{rhs1}
\bega
|\al|^3|\wtd u_\al(z)|&\lesssim \|w_\al\|_{\L^1_\rho}+ \|\pt_y\psi_\al\|_{\L^1_\rho}+\|\al\psi_\al\|_{\L^1_\rho}+\int_{c\delta_0}^\infty e^{-\frac{1}{2}|\al||y-z|}|w_\al(y)|dy\\
&\quad+\lambda \int_{c\delta_0}^\infty e^{-\frac{1}{2}|\al||y-z|}a(y)|\al||\psi_\al(y)|dy+\lambda \int_{c\delta_0}^\infty e^{-\frac{1}{2}|\al||y-z|}|a'(y)||\psi_\al(y)|dy\\
&\quad+\lambda \int_{c\delta_0}^\infty e^{-\frac{1}{2}|\al||y-z|} \al^2b(y)|\psi_\al(y)|dy .
\enda
\eeq
Now we will bound each term appearing on the right hand side of the above inequality.
Using Proposition \ref{near-bdr-prop}, we have 
\[\bega
\|\pt_y\psi_\al\|_{\L^1_\rho}+\|\al\psi_\al\|_{\L^1_\rho}&\lesssim \|\wtd u_\al\|_{\L^1_\rho}\lesssim \|\wtd u_\al\|_{\L^\infty_\rho}\lesssim \|w_\al\|_{\L^1_\rho}+\|yw_\al\|_{L^2(y\ge \delta_0+\rho)}\\
&\lesssim \|w_\al\|_{\L^1_\rho}+\|yw_\al\|_{L^2(y\ge c\delta_0)}.\\
\enda 
\]
Also, it is obvious that 
\[
\int_{c\delta_0}^\infty e^{-\frac{1}{2}|\al||y-z|}|w_\al(y)|dy\lesssim \|yw_\al\|_{L^2(y\ge c\delta_0)} .
\]
Now for the terms involving $a(y)$ on the right hand side of \eqref{rhs1}, we recall from the proof of \eqref{re-use2} that this term can be bounded by 
\[\bega
&\|\wtd u_\al\|_{\L^\infty_\rho}+\|a(y)\psi_\al\|_{L^\infty(y\ge \delta_0+\rho)}\lesssim \|w_\al\|_{L^1_\rho}+\|yw_\al\|_{L^2(y\ge \delta_0+\rho)}
\enda 
\]
thanks to Proposition \ref{near-bdr-prop} and Lemma \ref{lem-ellip2}. ~\\

Now for the last term on the right hand side of \eqref{rhs1}, we bound this term by
\[\bega
\lw\|
\lambda  \int_{c\delta_0}^\infty e^{-\frac{1}{2}|\al||y-z|} \al^2b(y)|\psi_\al(y)|dy
\rw\|_{\L^\infty_\rho}&\lesssim \lambda \|\nabla\psi_\al\|_{\L^\infty_\rho}+\lw\|\frac{\al \psi_\al}{1+\lambda y}\rw\|_{L^\infty(y\ge \delta_0+\rho)}\\
&\lesssim \|w_\al\|_{\L^1_\rho}+\|yw_\al\|_{L^2(y\ge \delta_0+\rho)} .
\enda 
\]
Here, we use the inequality \eqref{re-use1} and Lemma \ref{lem-ellip2}. The bound for the last term appearing in \eqref{rhs1} is complete. ~\\

Finally, combining the the bounds for all of the terms on the right hand side of \eqref{rhs1}, we get 
\[
|\al|^3|\wtd u_\al(z)|\lesssim \|w_\al\|_{\L^1_\rho}+\|yw_\al\|_{L^2(y\ge c\delta_0)}.
\]
Summing all $\al\in \lambda \mathbb{Z}$, we get 
\beq\label{ineq1}
\bega
\|\pt_x^3\wtd u\|_{L^\infty(\delta_1\le y\le \delta_2)}&\le \sum_{\al}|\al|^3\|\wtd u_\al\|_{L^\infty(\delta_1\le y\le \delta_2)}\\
&\lesssim \sum_{\al}|\al|^3\|\wtd u_\al\|_{L^\infty}\lesssim \|w\|_{\L^1_\rho}+\|yD_xw\|_{L^2(y\ge c\delta_0)}.
\enda
\eeq
On the other hand, for $\pt_y^3\wtd u$, we use $\partial_y^2 \psi = -\partial_x^2 \psi - \lambda L\psi+ \lambda^2 w$ to compute
\[
\bega 
\|\pt_y^3\wtd u\|_{L^\infty(\delta_1\le y\le \delta_2)}&\lesssim \|D_{x,y}^2w\|_{L^\infty(\delta_1\le y\le \delta_2)}+\|\pt_x^3\wtd u\|_{L^\infty(\delta_1\le y\le \delta_2)}\\
&\quad +\lambda \lw\|\pt_y L(\pt_x\psi)+L(\pt_x^2\psi)+\pt_y^2(L\psi)\rw\|_{L^\infty(\delta_1\le y\le \delta_2)}\\
&\lesssim \|w\|_{L^1_\rho}+\|yw\|_{L^2(y\ge c\delta_0)}\\
&\quad+\lambda \lw\|\pt_y L(\pt_x\psi)+L(\pt_x^2\psi)+\pt_y^2(L\psi)\rw\|_{L^\infty(\delta_1\le y\le \delta_2)}.
\enda 
\]
Using $L = a(y)\partial_y + b(y)\partial_x^2$, we thus obtain 
\[\bega
& \sum_{k\le 3}\|\pt_x^k\wtd u\|_{L^\infty(\delta_1\le y\le \delta_2)}+\|D_{x,y}^2w\|_{L^\infty(\delta_1\le y\le \delta_2)}\lesssim \|w\|_{\L^1_\rho}+\|yD_xw\|_{L^2(y\ge c\delta_0)}.
 \enda 
\]
The proof is complete.
\end{proof}

\subsection{Elliptic estimates away from the boundary}
We first show the following simple lemma that will be used in the next proposition.

\begin{lemma}\label{techLem}
Let $f(r),\xi(r)$ be smooth functions on $r\ge 1$, and $\xi(r)=0$ on $[1,R]$. Let $\phi$ solves the elliptic problem 
\[
\lw(\pt_r^2+\frac{1}{r}\pt_r-\frac{n^2}{r^2}\rw)\phi=\xi(r)\pt_r f(r)
\]
with the boundary condition $\phi|_{r=1}=0$.
There holds 
\[
|n|\lw\|\frac{\phi_n(r)}{r}\rw\|_{L^\infty}\lesssim \|\xi f\|_{L^\infty}+\|\xi' f\|_{L^1}
\]
\end{lemma}

\begin{proof}
As in \eqref{psi-n}, we get, for $n>0$
\[
-2n\phi_n(r)=\int_1^r\frac{s^{1+n}-s^{1-n}}{r^n}\xi(s)f'(s)ds+\int_r^\infty \lw(s^{1-n}r^n-\frac{s^{1-n}}{r^n}\rw)\xi(s)f'(s)ds.
\]
By integrating by parts, we get 
\[\bega 
-2n\phi_n(r)&=-\int_1^r \frac{(1+n)s^n-(1-n)s^{-n}}{r^n}\xi(s)f(s)ds-\int_1^r \frac{s^{1+n}-s^{1-n}}{r^n}\xi'(s)f(s)ds\\
&\quad +\frac{(r^{1+n}-r^{1-n})\xi(r)f(r)}{r^n}-\int_r^{\infty}\lw((1-n)s^{-n}r^n-(1-n)s^{-n}r^{-n}\rw)\xi(s)f(s)ds\\
&\quad -\int_r^\infty \lw(s^{1-n}r^n-\frac{s^{1-n}}{r^n}\rw)\xi'(s)f(s)ds.\\
\enda 
\]
Hence 
\[\bega 
n|\phi_n(r)|&\lesssim r \lw(\|\xi f\|_{L^\infty}+\|\xi' f\|_{L^1}\rw).
\enda 
\]
The proof is complete.
\end{proof}

Finally, we state the main Proposition for this section:

\begin{proposition}
Let $\psi$ be the solution to the elliptic problem \eqref{ellip-new}, and set $\wtd u=\nabla^\perp\psi$. For any $\delta\in (0,\delta_0)$, $k\ge 0$ and $\rho\in (\delta_0/4,\delta_0)$, one has 
\beq\label{D-k-new}
\bega
\|a(y)D^k_{x,y}\wtd u\|_{L^\infty (y\ge \delta)}&\lesssim \|w\|_{\L^1_\rho}+\|yD^{k+1}_{x,y}w\|_{L^2(y\ge \delta_0/2)},\\
\|D_{x,y}^k(a(y)\wtd u)\|_{L^2(y\ge \delta)}&\lesssim \|w\|_{\L^1_\rho}+\|yD_{x,y}^{k}w\|_{L^2(y\ge \delta_0/2)}.
\enda
\eeq
for $k\ge 0$, where $a(y)=\frac{1}{1+\lambda y}$.
\end{proposition}

\begin{proof} We first give the proof for $\|D^k_{x,y}\wtd u\|_{L^\infty(y\ge \delta)}$.
When $k=0$, the inequality \beq\label{re-use3}
\|a(y)\wtd u\|_{L^\infty}\lesssim \|w\|_{\L^1_\rho}+\|yD_xw\|_{L^2(y\ge \delta_0/2)}
\eeq
follows from Lemma \ref{lem-ellip2}. Moreover, we also have 
\[\bega 
\|a(y)\pt_x^k\wtd u\|_{L^\infty}&\lesssim \|\pt_x^kw\|_{\L^1_{\delta_0/8}}+\|yD_x^{k+1}w\|_{L^2(y\ge \delta_0/2)}\\
&\lesssim \|w\|_{\L^1_\rho}+\|yD_x^{k+1}w\|_{L^2(y\ge \delta_0/2)}\enda 
\]
where we use the fact that $\rho\ge \frac{\delta_0}{8}$. We can now assume that $D^k_{x,y}=\pt_y^k$ and we will use induction on $k\ge 0$. 
We first give a proof for $k=1$, which is $\pt_y$. 
We have 
\beq\label{diff}
\begin{cases}
&\pt_y(\pt_x\psi)=\pt_x(\pt_y\psi)\\
&\pt_y(\pt_y\psi)=\pt_y^2\psi=\lambda^2 w-\pt_x^2\psi-\lambda a(y)\pt_y \psi-\lambda b(y)\pt_x^2\psi.
\end{cases}
\eeq
For the first term $\pt_y(\pt_x\psi)$, we simply bound 
\[
\|\pt_x\pt_y\psi\|_{L^\infty(y\ge \delta)}\le \|\pt_x\wtd u\|_{L^\infty(y\ge \delta)}\lesssim \|\pt_xw\|_{\L^1_\rho}+\|yD_x^2w\|_{L^2(y\ge \delta_0/2)}
\]
For the second term $\pt_y^2\psi$ in  \eqref{diff} we get, for any $y\ge\delta$: 
\[\bega 
|a(y)\pt_y^2\psi_\al(y)|&\le a(y)|w_\al(y)|+a(y)|\al|^2|\psi_\al(y)|+a(y)|\pt_y\psi_\al(y)|+\lambda a(y)b(y)|\al|^2|\psi_\al(y)|\\
&\lesssim \|w_\al\|_{L^\infty(y\ge \delta)}+\|a(y)\al \wtd u_\al\|_{L^\infty}+\|a(y)\wtd u_\al\|_{L^\infty}\\
&\lesssim \|w_\al\|_{L^\infty(y\ge \delta)}+\| w_\al\|_{\L^1_\rho}+\|y\al w_\al\|_{L^2(y\ge \delta_0/2)}
.\enda 
\]
Let $\zeta$ be a cut-off function so that 
\beq\label{cutoff}
\zeta(y)=
\begin{cases}
0&\qquad y\le \delta/2,\\
1&\qquad y\ge \delta.
\end{cases}
\eeq
Then we have 
\[
\|a(y)w_\al\|_{L^\infty(y\ge \delta)}\lesssim \|w_\al\|_{L^\infty(y\ge \delta)}\le \|\zeta (z)w_\al(z)\|_{L^\infty}.
\]
We have 
$$ \zeta(z)w_\al(z) =\int_0^z \zeta'(y)w_\al(y)dy+\int_0^z \zeta (y)\pt_yw_\al(y)dy.
$$
Hence, for every $z\ge 0$, we bound
\[\bega 
|\zeta (z)w_\al(z)|&\le \|w_\al\|_{L^\infty(\delta/2\le y\le \delta)}+\int_{\delta/2}^\infty |\pt_y w_\al(z)|dz\\
&\lesssim \|w_\al\|_{\L^1_\rho}+\|\pt_yw_\al\|_{L^\infty(\delta/2\le y \le \delta_0/2)}+\int_{\delta_0/2}^\infty |\pt_yw_\al(z)|dz\\
&\lesssim \|w_\al\|_{\L^1_\rho}+\|y\pt_yw_\al\|_{L^2(y\ge \delta_0/2)}.
\enda 
\]
Combining the above inequalities, we obtain 
\[\bega
\|a(y)\pt_y^2\psi_\al\|_{L^\infty(y\ge \delta)}&\lesssim \|w\|_{\L^1_\rho}+\sum_{\al\in \lambda \Z}\lw(\|y\pt_yw_\al\|_{L^2(y\ge \delta_0/2)}+\|y\al w_\al\|_{L^2(y\ge \delta_0/2)}\rw)\\
&\lesssim \|w\|_{\L^1_\rho}+\|yD_{x,y}^2w\|_{L^2(y\ge \delta_0/2)}.
\enda \]
This finishes the proof for $k=1$. 
Now we assume $k\ge 1$. We proceed by induction on the number of derivatives of $y$. Assume that the inequality is true for $k-1$, we show that it is also true for $k$. 
We recall that $\zeta$ be a cut-off function defined in \eqref{cutoff}. Then $\zeta \pt_y^k \psi$ solves the elliptic problem 
\beq\label{deri}
\bega 
(\triangle+\lambda L) (\zeta \pt_y^k\psi)=&\lambda^2 (\zeta \pt_y^kw)+2\zeta'(y)\pt_y\pt_y^k\psi+\zeta''(y)\pt_y^k\psi\\
&+\lambda L(\zeta \pt_y^k\psi)-\lambda \zeta \pt_y^k(L\psi)
\enda 
\eeq
with the boundary condition $\zeta \pt_y^k\psi|_{y=0}=0$.
In Fourier frequency $\al$, the right hand side in the above can be decomposed into $F_1+F_2+F_3$ where 
\[
\begin{cases}
F_1&=\lambda^2 (\zeta \pt_y^kw_\al)+2\zeta'(y)\pt_y\pt_y^k\psi+\zeta''(y)\pt_y^k\psi_\al,\\
F_2&=\lambda a(y)\pt_y(\zeta\pt_y^k\psi_\al)-\lambda \zeta \pt_y^k(a(y)\pt_y\psi_\al),\\
F_3&=-\lambda b(y)\al^2  (\zeta \pt_y^k\psi_\al)+\lambda \zeta \pt_y^k(b(y)\al^2\psi_\al).\\\end{cases}
\]
From the equation \eqref{deri}, we get 
\[
\zeta \pt_y^k\psi=\Psi_1+\Psi_2+\Psi_3
\]
where $(\triangle+\lambda L)\Psi_i=F_i$ for $1\le i\le 3$ with the boundary condition $\Psi_i|_{y=0}=0$ (this can also be seen from the formula \eqref{psi-n}). We also denote 
\begin{equation}\label{def-UUUi}U_i=\nabla^\perp \Psi_i.\end{equation}

~\\
\textbf{Treating $U_1$}.~
Using the same argument as in Lemma \ref{lem-ellip2}, for every $z\ge 0$, we get 
\[\bega 
\lw |a(z) U_1(z)\rw |&\lesssim \int_0^\infty \lw(|\zeta \pt_y^kw_\al(y)|+|\zeta'(y)\pt_y^k\wtd u_\al(y)|+|\zeta''(y)||\pt_y^{k-1}\wtd u_\al(y)|\rw)dy\\
&\lesssim \|y\zeta \pt_y^kw_\al\|_{L^2}+\|\zeta'(y)D^k_{x,y}\wtd u_\al\|_{L^1}\lesssim \|yD^k_{x,y}w_\al\|_{L^2(y\ge \delta/2)}+\|D^k_{x,y}\wtd u_\al\|_{L^1(\delta/2\le y\le \delta)}\\
&\lesssim \|yD^k_{x,y}w_\al\|_{L^2(\delta/2\le y\le \delta_0/2)}+\|yD^k_{x,y}w_\al\|_{L^2(y\ge\delta_0/2)}+\lw(\|w\|_{\L^1_\rho}+\|yw_\al\|_{L^2(y\ge \delta_0/2)}\rw)\\
&\lesssim \|w_\al\|_{\L^1_\rho}+ \|yD^k_{x,y}w_\al\|_{L^2(y\ge \delta_0/2)}.
\enda \]

~\\
\textbf{Treating $U_2$.}~
We have 
\[\bega 
F_2&=-\lambda \sum_{1\le i\le k}{k \choose i}\zeta(y)\pt_y^ia(y)\pt_y^{k+1-i}\psi_\al+\lambda a(y)\zeta'(y)\pt_y^k\psi_\al\\
&=-\lambda \sum_{2\le i\le k}{k \choose i}\pt_y^i a(y)\zeta(y)\pt_y^{k+1-i}\psi_\al-\lambda a'(y)\zeta(y)\pt_y^{k}\psi_\al+\lambda a(y)\zeta'(y)\pt_y^k\psi_\al\\
&=F_{2,1}+F_{2,2}+F_{2,3}.
\enda 
\]
Hence we get $U_2=\sum_{i=1}^3 (\triangle+\lambda L)^{-1}F_{2,i}=\sum_{i=1}^3U_{2,i}$. Arguing as in Lemma \eqref{lem-ellip2}, for every $z\ge 0$, we get 
\[\bega 
a(z)U_{2,1}(z)&\lesssim\max_{2\le i\le k} \int_0^\infty |\pt_y^i a(y)\zeta (y)\pt_y^{k-1}\wtd u_\al(y)|dy\\
&\lesssim \max_{2\le i\le k}\int_0^\infty a(y)^{i+1}\zeta (y)|\pt_y^{k-1}\wtd u_\al(y)|dy\\
&\lesssim \|a(y)\zeta (y)\pt_y^{k-1}\wtd u_\al(y)\|_{L^\infty} \max_{2\le i\le k}\int_0^\infty a(y)^idy\\
&\lesssim \|a(y)\zeta (y)\pt_y^{k-1}\wtd u_\al\|_{L^\infty}\lesssim \|w\|_{\L^1_\rho}+\|yD_{x,y}^{k-1}w_\al\|_{L^2(y\ge\delta_0/2)}.
\enda 
\]
Here, we have used the fact that $\pt_y^ia(y)\lesssim a(y)^{i+1}$, $\int_0^\infty a(y)^idy\lesssim 1$ for all $i\ge 2$, and the induction hypothesis in the last inequality.

Now we turn to $F_{2,2}=-\lambda a'(y)\zeta(y)\pt_y^k\psi_\al$. 
Applying Lemma \ref{techLem} for $\xi(y)=-\lambda a'(y)\zeta(y)$ and $f(y)=\pt_y^{k-1}\psi_\al$, for every $z\ge 0$, we get 
\[\bega 
a(z)U_{2,2}(z)&\lesssim \|a'(y)\zeta (y)\pt_y^{k-1}\psi_\al\|_{L^\infty}+\|\pt_y(a'(y)\zeta(y))\pt_y^{k-1}\psi_\al\|_{L^1}\\
&\lesssim \|a(y)\zeta(y) \pt_y^{k-1}\psi_\al\|_{L^\infty}+\|a''(y)\zeta(y)\pt_y^{k-1}\psi_\al\|_{L^1}+\|a'(y)\zeta'(y)\pt_y^{k-1}\psi_\al\|_{L^1}\\
&\lesssim \|w_\al\|_{\L^1_\rho}+\|yD_{x,y}^{k-1}w_\al\|_{L^2(y\ge\delta_0/2)}+\|a''(y)\|_{L^1}\|a(y)\zeta (y)\pt_y^{k-1}\psi_\al\|_{L^\infty}+\|\pt_y^{k-1}\psi_\al\|_{L^\infty(\delta/2\le y\le \delta)}\\
&\lesssim \|w_\al\|_{\L^1_\rho}+\|yD_{x,y}^{k-1}w_\al\|_{L^2(y\ge \delta_0/2)}.
\enda 
\]

Finally, for $U_{2,3}$ which solves $(\triangle +\lambda L)U_{2,3}=F_{2,3}=\lambda a(y)\zeta'(y)\pt_y^{k}\psi_\al$, we use Lemma \ref{lem-ellip2} again, for every $z\ge 0$, to get 
\[
a(z)U_{2,3}(z)\lesssim \int_0^\infty \lambda a(y)|\zeta'(y)\pt_y^{k}\psi_\al(y)|dy\lesssim \|\pt_y^{k}\psi_\al\|_{L^\infty(\delta/2\le y\le \delta)}\lesssim \|w_\al\|_{\L^1}+\|yw_\al\|_{L^2(y\ge \delta_0/2)}.
\]
The proof for $U_2$ is complete.

~\\
\textbf{Treating $U_3$}.~We recall that 
\[\bega 
F_3&=-\lambda b(y)\al^2  (\zeta \pt_y^k\psi_\al)+\lambda \zeta \pt_y^k(b(y)\al^2\psi_\al)\\
&= \lambda \al^2 \sum_{i=1}^k {k \choose i} \zeta(y)\pt_y^ib(y)\pt_y^{k-i}\psi_\al(y).
\enda 
\]
Using Lemma \ref{lem-ellip2} for the equation $(\triangle+\lambda L)\Psi_3=F_3$, for $z\ge 0$, we get
\[\bega 
|a(z)U_3(z)|&\lesssim \max_{1\le i\le k} \lambda \al^2 \int_0^\infty \zeta(y)|\pt_y^ib(y)||\pt_y^{k-i}\psi_\al(y)|dy\\
&\lesssim \max_{1\le i\le k}|\al|\int_0^\infty a(y)^{i+2}\zeta(y)|\al \pt_y^{k-i}\psi_\al(y)|dy\\
&\lesssim \|a(y)\zeta(y) D_{x,y}^{k-1}\wtd u_\al\|_{L^\infty}\max_{1\le i\le k}\int_0^\infty a(y)^{i+1}dy\\
&\lesssim \|w_\al\|_{\L^1_\rho}+\|yD_{x,y}^{k-1}w_\al\|_{L^2(y\ge \delta_0/2)}.
\enda 
\]
where we use the induction hypothesis in the last inequality, and the fact that $\pt_y^i b(y)\lesssim a(y)^{i+2}$ for all $i\ge 1$. The proof is complete for the $\|\cdot\|_{L^\infty(y\ge \delta)}$ norm of the velocity. The estimates in $L^2$ norm follow similarly. 
\end{proof}

\section{Bilinear estimates}
In this section, we recall the bilinear estimates for the nonlinear terms. 
We define the nonlinear quantity for $w$ as follows:
\beq \label{nonlinear N}
\bega 
N_\rho(w,k)=&\|w\|_{\mathcal W^{k+1,1}_\rho}\lw(\|w\|_{\mathcal W^{k,1}_\rho}+\|yD_{x,y}^{k}w\|_{L^2(z\ge \delta_0 +\rho)}\rw)\\
&+\|w\|_{\mathcal W^{k,1}_\rho}\|yD_{x,y}^{k+2}w\|_{L^2(z\ge \delta_0 +\rho)}.
\enda 
\eeq

\begin{proposition} \label{bilinear1}Let $N_\rho(w,k)$ be the nonlinear quantity defined in \eqref{nonlinear N}, and $\psi=(\triangle+\lambda L)^{-1}(\lambda^2w)$ be the corresponding stream function defined in the elliptic problem \eqref{ellip-new}. 
For $k\in \{0,1\}$, there hold
\[
\|\wtd u\cdot \nabla w\|_{\mathcal W_\rho^{k,1}}\lesssim N_\rho(w,k)
\]
where $\wtd u=\nabla^\perp \psi$.
\end{proposition}

\begin{proof} For $k=0$, 
we have 
\[\bega 
\|\pt_x\psi\pt_yw\|_{\mathcal L_\rho^1}&\le \|y^{-1}\pt_x\psi\|_{\L^\infty_\rho}\|y\pt_yw\|_{\L^1_\rho}\\
&\lesssim \lw(\|w\|_{\mathcal L_\rho^1}+\|\pt_x w\|_{\mathcal L_\rho^1}+\|yD_{x,y}^2w\|_{L^2(y\ge \delta_0+\rho)}\rw)\|y\pt_yw\|_{\mathcal L^1_\rho}.
\enda 
\]
upon using Proposition \ref{near-bdr-prop}. Similarly, for $k=1$, we compute 
\beq \label{com}
\begin{cases}
&\pt_x(\pt_x\psi\pt_yw)=y^{-1}\pt_x^2\psi\cdot y\pt_yw+y^{-1}\pt_x\psi\cdot\pt_x(y\pt_yw)\\
&y\pt_y(\pt_x\psi\pt_yw)=\pt_x(\pt_y\psi)\cdot y\pt_yw+y^{-1}\pt_x\psi\cdot \lw\{(y\pt_y)^2w-y\pt_yw\rw\}.
\end{cases}
\eeq
This implies 
\[\bega
\|\pt_x(\pt_x\psi\pt_yw)\|_{\mathcal L_\rho^1}
&\lesssim \|y^{-1}\pt_x^2\psi\|_{\L^\infty_\rho}\|y\pt_yw\|_{\L^1_\rho}+\|y^{-1}\pt_x\psi\|_{\L^\infty_\rho}\|\pt_x(y\pt_yw)\|_{\L^1_\rho}\\
&\lesssim \lw(\|\pt_xw\|_{\mathcal L_\rho^1}+\|\pt_x^2 w\|_{\mathcal L_\rho^1}+\|yD_{x,y}^3w\|_{L^2(y\ge \delta_0+\rho)}\rw)\|y\pt_yw\|_{\mathcal L^1_\rho}\\
&\quad +\lw(\|w\|_{\mathcal L_\rho^1}+\|\pt_x w\|_{\mathcal L_\rho^1}+\|yD_{x,y}^2w\|_{L^2(y\ge \delta_0+\rho)}\rw)\|\pt_x (y\pt_yw)\|_{\L^1_\rho}\\
&\lesssim \|w\|_{\mathcal W^{2,1}_\rho}\|w\|_{\mathcal W^{1,1}_\rho}+\|yD_{x,y}^3w\|_{L^2(y\ge \delta_0+\rho)}\|w\|_{\mathcal W^{1,1}_\rho}+\|yD^2_{x,y}w\|_{L^2(y\ge \delta_0+\rho)}\|w\|_{\mathcal W^{2,1}_\rho}.
\enda 
\]
Similarly, from the calculation in \eqref{com}, we have \[\bega 
&\|y\pt_y(\pt_x\psi\pt_yw)\|_{\L^1_\rho}\\
&\lesssim  \|\pt_x\wtd u\|_{\L^\infty_\rho}\|y\pt_yw\|_{\L^1_\rho}+\|y^{-1}\pt_x\psi\|_{\L^\infty_\rho}\|w\|_{\mathcal W^{2,1}_\rho}\\
&\lesssim \lw(\|w\|_{\mathcal W^{1,1}_\rho}+\|yD_{x,y}^2w\|_{L^2(y\ge \delta_0+\rho)}\rw)\|w\|_{\mathcal W_\rho^{1,1}}+\lw(\|w\|_{\mathcal W^{1,1}_\rho} +\|yD_{x,y}^{2}w\|_{L^2(y\ge \delta_0+\rho)} \rw)\|w\|_{\mathcal W^{2,1}_\rho},\\
\enda 
\]
giving the proposition. 
\end{proof}

Next we show the nonlinear estimate away from the boundary:

\begin{lemma}\label{bilinear2} There holds
\[\bega 
&\|yD^2_{x,y}(a(y)\wtd u\cdot\nabla w)\|_{L^2(y\ge \delta_0+\rho)}\lesssim \|yD^3_{x,y}w\|_{L^2(y\ge \delta_0/2)}\lw(  \|w\|_{\L^1_\rho}+\|yD_{x,y}^3w\|_{L^2(y\ge \delta_0/2)}\rw)\\
\enda 
\]
\end{lemma}
\begin{proof}
We give the proof for the case when there is no derivative only. The other cases are treated similarly. 
We have 
\[\bega
\|ya(y)\wtd u\pt_yw\|_{L^2(y\ge \delta_0+\rho)}&=\|a(y)\wtd u_2 \|_{L^\infty(y\ge\delta_0+\rho)}\|y\pt_yw\|_{L^2(y\ge\delta_0+\rho)}+\|a(y)\wtd u_1\|_{L^\infty(y\ge \delta_0+\rho)}\|\pt_xw\|_{L^2(y\ge \delta_0+\rho)}\\
&\lesssim \lw(\|w\|_{\L^1_\rho}+\|yD^1_{x,y}w\|_{L^2(y\ge \delta_0/2)}\rw)\|yD_{x,y}^1w\|_{L^2(y\ge \delta_0/2)}.
\enda
\]
where we used \ref{D-k-new}. The proof is complete.
\end{proof}

\section{Estimates for vorticity away from the boundary}\label{energy-sec}
In this section, we estimate 
\[
\|y^2 D_{x,y}^5w\|_{L^2(y\ge \delta_0/2)}=\sum_{i+j\le 5}\|y^2 \pt_x^i\pt_y^jw\|_{L^2(y\ge \delta_0/2)}
\]
for the scaled vorticity $w$ solving \eqref{new-eq}. 
We take a cut off function $\eta: [0,\infty)\to [0,\infty)$ such that 
\beq \label{cutoff}
\eta(y)=\begin{cases} 
0&\quad \text{if}\quad y\le \delta_0/4\\
y^2&\quad \text{if}\quad y\ge \delta_0/2.
\end{cases}
\eeq
We define 
\[
\mathcal E(t)=\sum_{i+j\le 5}\frac{1}{2}\int_0^\infty \eta (y) |\pt_x^i\pt_y^j w(t)|^2dy
\]
to be the main control for the norm $\|y^2D_{x,y}^5w\|_{L^2(y\ge \delta_0/2)}$, and 
\[
\mathcal D(t)=\sum_{i+j\le 5} \nu\lw\{ 
\int(1+\lambda b(y))\eta(y)|\pt_x^{i+1}\pt_y^jw|^2
+\frac{1}{2}\int \eta'(y)|\pt_x^i\pt_y^{j+1}w|^2
\rw\}
\]
coming from the dissipation term in the energy estimate. We note that $\eta'(y)>0$, so all the terms in $\mathcal D (t)$ are non-negative.
Moreover, we define the following quantity away from the boundary that is needed to bound $\mathcal E(t)$:
\beq\label{Na}
\bega
N_a(\wtd u,w)=&\|D_{x,y}^4(a(y)\wtd u)\|_{L^\infty(y\ge \delta_0/4)}+\|D^5_{x,y}(a(y)\wtd u)\|_{L^2(y\ge \delta_0/4)}+\|D_{x,y}^5w\|_{L^\infty(\delta_0/4\le y\le \delta_0/2)}.
\enda
\eeq

We obtain the following proposition.
 
\begin{proposition} \label{energy} Let $(w,\psi)$ solve \eqref{new-eq}-\eqref{def-streamscaled}, and set  $\wtd u=\nabla^\perp\psi$. For $\lambda$ sufficiently small, there holds 
\[\bega 
\mathcal E'(t)+c_0 \mathcal D(t)&\le C_0 \lw(\mathcal E(t)+N_a(\wtd u,w)\mathcal E(t)+N_a(\wtd u,w)^2+N_a(\wtd u,w)^2\mathcal E(t)^{1/2}\rw)
\enda 
\]
for some constants $c_0,C_0>0$.
\end{proposition}

\begin{proof} Using \eqref{new-eq}, we compute
\[\bega 
\mathcal E'(t)=&\sum_{i+j\le 5}\int_0^\infty \eta (y)\pt_x^i\pt_y^j \pt_t w\cdot \pt_x^i\pt_y^jw\\
=&\sum_{i+j\le 5}\nu \int\eta \triangle (\pt_x^i\pt_y^jw)\cdot\pt_x^i\pt_y^j w+\nu\lambda  \int \eta(y) \pt_x^i\pt_y^j(a(y)\pt_yw)\cdot\pt_x^i\pt_y^jw\\
&+\nu\lambda \int \eta (y)\pt_x^i\pt_y^j(b(y)\pt_x^2w)\cdot\pt_x^i\pt_y^jw+\int \eta(y)\pt_x^i\pt_y^j\lw(a(y)\wtd u\cdot\nabla w\rw)\cdot\pt_x^i\pt_y^j w\\
=&\sum_{k=1}^7\mathcal I_k
\enda 
\]
where 
\[\begin{cases}
\mathcal I_1&=\sum_{i+j\le 5}\nu \int\eta \triangle (\pt_x^i\pt_y^jw)\cdot\pt_x^i\pt_y^j w\\
\mathcal I_2&=\sum_{i+j\le 5}\nu\lambda  \int \eta(y) (a(y)\pt_x^i\pt_y^{j+1}w)\cdot\pt_x^i\pt_y^jw\\
\mathcal I_3&=\sum_{i+j\le 5}\nu\lambda  \int \eta(y) \lw\{\pt_x^i\pt_y^j(a(y)\pt_yw)-(a(y)\pt_x^i\pt_y^{j+1}w)\rw\}\cdot\pt_x^i\pt_y^jw\\
\mathcal I_4&=\sum_{i+j\le 5}\nu\lambda  \int \eta(y) b(y)\pt_x^{i+2}\pt_y^j w\cdot\pt_x^i\pt_y^jw\\
\mathcal I_5&=\sum_{i+j\le 5}\nu\lambda  \int \eta(y) \lw\{\pt_x^i\pt_y^j(b(y)\pt_x^2w)-b(y)\pt_x^{i+2}\pt_y^jw)\rw\}\cdot\pt_x^i\pt_y^jw\\
\mathcal I_6&=\sum_{i+j\le 5} \int  \eta(y)\lw(a(y)\wtd u\cdot\nabla \pt_x^i\pt_y^jw\rw)\cdot\pt_x^i\pt_y^j w\\
\mathcal I_7&=\sum_{i+j\le 5} \int  \eta(y)\lw(\pt_x^i\pt_y^j\lw(a(y)\wtd u\cdot\nabla w\rw)-a(y)\wtd u\cdot\nabla \pt_x^i\pt_y^jw\rw)\cdot\pt_x^i\pt_y^j w\\
\end{cases}
\]
Below, we sometimes skip writing $\sum_{i+j\le 5}$, without any confusion.~\\
By integrating by parts, we obtain 
\[\bega 
\mathcal I_1=&-\nu \int \eta |\pt_x^{i+1}\pt_y^jw|^2-\nu \int \pt_x^i\pt_y^{j+1}w\pt_y(\eta \pt_x^i\pt_y^{j+1}w)\\
=&-\nu \int \eta |\pt_x^{i+1}\pt_y^jw|^2-\nu \int \eta'(y)|\pt_x^i\pt_y^{j+1}w|^2-\nu \int \eta (y)\pt_x^i\pt_y^{j+1}w\cdot \pt_x^i\pt_y^{j+2}w ,
\enda 
\]
which yields
\[
\mathcal I_1=-\nu \int \eta |\pt_x^{i+1}\pt_y^jw|^2-\frac{1}{2}\nu \int \eta'(y)|\pt_x^i\pt_y^{j+1}w|^2.
\]
Similarly, we get
\[
\mathcal I_4=-\nu \lambda \int \eta(y)b(y)|\pt_x^{i+1}\pt_y^jw|^2.
\]
On the other hand, we will now show that $\mathcal I_2+\mathcal I_3\lesssim \mathcal E(t)$.
Indeed, by integrating by parts, we have 
\[
\mathcal I_2=-\nu \lambda \frac{1}{2}\sum_{i+j\le 5}\int |\pt_x^i\pt_y^jw|^2\pt_y (\eta (y)a(y))\lesssim \sum_{i+j\le 5} \int \eta (y)|\pt_x^i\pt_y^jw|^2,
\]
and for $\mathcal I_3$, we have 
\[
\mathcal I_3\lesssim \nu \lambda \sum_{i+j\le 5} \int\eta(y)||\pt_x^i\pt_y^jw|^2\lesssim \mathcal E(t).
\]
For $\mathcal I_5$, we use integration by parts in $x$ to get 
\[
\mathcal I_5\lesssim \nu \lambda \sum_{i+j\le 5} \int \eta (y)|\pt_x^{i+1}\pt_y^jw|^2\lesssim \lambda \mathcal D(t).
\]
For $\mathcal I_6$, we have
\[\bega
\mathcal I_6=&\sum_{i+j\le 5}\frac{1}{2}\int \eta (y)a(y)\wtd u\cdot\nabla \lw(|\pt_x^i\pt_y^jw|^2\rw)=-\sum_{i+j\le 5}\frac{1}{2}\int \div(\eta(y)a(y)\wtd u)|\pt_x^i\pt_y^jw|^2\\
=&-\sum_{i+j\le 5}\frac{1}{2}\int \wtd u\cdot\nabla(\eta(y)a(y))|\pt_x^i\pt_y^jw|^2=-\frac{1}{2}\sum_{i+j\le 5}\int \wtd u_2\pt_y(\eta(y)a(y))|\pt_x^i\pt_y^jw|^2.
\enda
\]
If $y\ge \delta_0/2$ then we have 
\[
|\eta'(y)|=2y\lesssim y^2=\eta(y).
\]
Hence 
\[
-\sum_{i+j\le 4}\frac{1}{2}\int_{\delta_0}^\infty \wtd u_2\pt_y(\eta(y)a(y))|\pt_x^i\pt_y^jw|^2\lesssim \|a(y)\wtd u_2\|_{L^\infty(y\ge \delta_0/4)}\mathcal E(t).
\]
When $\frac{\delta_0}{4}\le y\le \frac{\delta_0}{2}$, we get 
\[\bega 
-\sum_{i+j\le 5}\frac{1}{2}\int_{\delta_0/4}^{\delta_0/2} \wtd u_2\pt_y(\eta(y)a(y))|\pt_x^i\pt_y^jw|^2\lesssim \|\wtd u_2\|_{L^\infty(\delta_0/4\le y\le \delta_0/2)}\lw(\mathcal E(t)+\sum_{i+j\le 5}\|\pt_x^i\pt_y^jw\|_{L^2(\delta_0/4\le y\le \delta_0/2)}
\rw).
\enda 
\]
This implies that
\[
\mathcal I_6\lesssim N_a(\wtd u,w)\mathcal E(t)+N_a(\wtd u,w)^2.
\]
Lastly, we have 
\[\bega
\mathcal I_7&\lesssim \|D^4_{x,y}(a(y)\wtd u)\|_{L^\infty(y\ge \delta_0/4)}\mathcal E(t)+\|D^5_{x,y}(a(y)\wtd u)\|_{L^2(y\ge \delta_0/4)}\|\eta(y)^{1/2}\nabla w\|_{L^\infty}\mathcal E(t)^{1/2}.
\enda
\]
Using the Sobolev embedding $L^\infty(\mathbb T\times \mathbb R)\subset H^{4}(\mathbb T\times \mathbb R)$, we have 
\[
\|\eta(y)^{1/2}\nabla w\|_{L^\infty}\lesssim \|D^4_{x,y}(\eta^{1/2}\nabla w)\|_{L^2}\lesssim \mathcal E(t)^{1/2}+\|D^5_{x,y}w\|_{L^\infty(\delta_0/4\le y\le \delta_0/2)}.
\]
The proof is complete.
\end{proof}

\begin{proposition}\label{Na}
There holds 
\[
N_a(\wtd u,w)\lesssim \|w\|_{\L^1_\rho}+\|yD_{x,y}^5w\|_{L^2(y\ge \delta_0/2)}
\]
\end{proposition}

\begin{proof}
This is a direct consequence of the inequality \eqref{D-k-new} and Lemma \ref{middle}. The proof is complete.
\end{proof}
\section{Nonlinear analysis}\label{sec7}
Our goal in this section is to combine all the estimates in analytic norm and Sobolev norms in the previous sections.
We recall that $w$ is the solution to the problem 
\[
\bega 
(\pt_{\tau }-\nu \triangle-\nu \lambda L)w&=f,\\
\nu(\pt_y+N)w|_{y=0}&=g,
\enda 
\]
where 
\[
\bega 
f&=a(y)\wtd u\cdot\nabla w,\quad \wtd u=-\nabla^\perp\psi,\\
g &=-\pt_y (\triangle+\lambda L)^{-1}f|_{y=0}.\\
\enda 
\]
We will use the coupled semigroup estimate for the exterior domain \ref{coupled}. 
We also recall the quantity defined in \eqref{nonlinear N}:
\[
\bega 
N_\rho(w,k)=&\|w\|_{\mathcal W^{k+1,1}_\rho}\lw(\|w\|_{\mathcal W^{k,1}_\rho}+\|yD_{x,y}^{k}w\|_{L^2(y\ge \delta_0 +\rho)}\rw)\\
&+\|w\|_{\mathcal W^{k,1}_\rho}\|yD_{x,y}^{k+2}w\|_{L^2(y\ge \delta_0 +\rho)},
\enda 
\]
First we show the semigroup estimates. 

\begin{proposition}\label{g-bound} Let $0\le k\le 2$, there holds \[
\|g(s)\|_{\mathcal H^{k}_\rho}\lesssim N_\rho(w(s),k)+\|yD_{x,y}^{k+2}w(s)\|^2_{L^2(y\ge \delta_0/2)} .
\]
\end{proposition}

\begin{proof}
We define the function $p$ solving the elliptic problem $(\triangle+\lambda L)p=a(y)\wtd u\cdot\nabla w$ with the boundary condition $p|_{y=0}=0$. We have 
\[\bega 
\sum_\al e^{\eps_0(\delta_0+\rho)|\al|}|g|&=\sum_{\al}e^{\eps_0(\delta_0+\rho)|\al|}|\pt_y p_\al(0)|\lesssim  \|\pt_yp\|_{\L^\infty_\rho}\\
&\lesssim \|a(y)\wtd u\cdot\nabla w\|_{\L^1_\rho}+\|ya(y)D_x\lw(\wtd u\cdot\nabla w\rw)\|_{L^2(y\ge \delta_0+\rho)}\\
&\lesssim N_\rho(w,0)+\|a(y)\wtd u\|_{L^\infty}\|yD^2_{x,y}w\|_{L^2(y\ge \delta_0+\rho)}+\|a(y)D_x\wtd u\|_{L^\infty(y\ge \delta_0+\rho)}\|yD_{x,y}^1w\|_{L^2(y\ge \delta_0/2)}\\
&\lesssim N_\rho(w,0)+\lw(\|w\|_{\L^1_\rho}+\|yD_{x,y}^1w\|_{L^2(y\ge \delta_0/2)}\rw)\|yD^2_{x,y}w\|_{L^2(y\ge \delta_0/2)}\\
\enda 
\]
where we use Proposition \ref{bilinear1}. The proof is complete.
\end{proof}
Now we give the proof for our main theorem. Using the coupled semigroup estimates \ref{coupled}, we define the norm for $1\le k\le 3$ (we can take $k=1$).
\beq\label{main-norm}
\bega 
\mathcal A_k(w(\tau),\rho)&=\lw(\|w(\tau )\|_{\mathcal W^{k,1}_\rho}+ \sqrt{\nu \tau} \|w(\tau)|_{z=0}\|_{\mathcal H^{k-1}_\rho}\rw)\\
&\quad+\lw(\|w(\tau )\|_{\mathcal W^{k+1,1}_\rho}+ \sqrt{\nu \tau}\|w(\tau)|_{z=0}\|_{\mathcal H_\rho^{k}}\rw)(\rho_0-\rho-\beta \tau )^{\gamma}\\
\enda 
\eeq
and the quantity 
\[\bega
A(\beta)=&\sup_{0<\tau  \beta <\rho_0}\quad \lw\{\sup_{0<\rho<\rho_0-\beta \tau }\lw(\mathcal A_k(w(\tau),\rho) \rw)\rw\}+\sup_{0<\tau \beta<\rho_0}\|yD_{x,y}^3w\|_{L^2(y\ge \delta_0/2)}.
\enda 
\]

\begin{proposition}\label{iter-beta} There holds 
\[\bega
A(\beta)\lesssim&\|w_0\|_{\mathcal W^{2,1}_{\rho_0}}+\|yD_{x,y}^5w_0\|_{L^2(y\ge \delta_0/4)}+\beta^{-1}A(\beta)^2 \\
&+e^{C(1+A(\beta))\beta^{-1}}\lw(\|yD_{x,y}^5w_0\|_{L^2(y\ge \delta_0/4)}+\beta^{-1}A(\beta)^2\rw).
\enda
\]
\end{proposition}

\begin{proof}
For simplicity, we let 
\[
M_0=\|w_0\|_{\mathcal W^{k+1,1}_{\rho_0}}+\|yD_{x,y}^5w_0\|_{L^2(y\ge \delta_0/2)}
\]
First we bound $\|w(\tau)\|_{\mathcal W^{k,1}_\rho}$. By Theorem \ref{coupled} and Proposition \ref{g-bound}, we get 
\[\bega 
\|w(\tau)\|_{\mathcal W^{k,1}_\rho}&\lesssim M_0+\lambda A(\beta)+\lambda\nu \int_0^\tau A(\beta)(\rho_0-\rho-\beta s)^{-\gamma}ds\\
&+\sqrt {\lambda \nu} A(\beta)\int_0^\tau (\rho_0-\rho-\beta s)^{-\gamma}ds
+A(\beta)^2\int_0^\tau (\rho_0-\rho-\beta s)^{-\gamma}ds+\beta^{-1}A(\beta)\\
&\lesssim M_0+\lambda A(\beta)+ \beta^{-1}A(\beta)+\beta^{-1}A(\beta)^2.
\enda 
\]
Next, we bound $ \sqrt{\nu \tau} \|w(\tau)|_{z=0}\|_{\mathcal H^{k-1}_\rho}$. From Propositions \ref{at-bdr} and \ref{g-bound}, we get 
\[\bega 
 \sqrt{\nu \tau} \|w(\tau)|_{z=0}\|_{\mathcal H^{k-1}_\rho}&\lesssim M_0+\lambda \nu\sqrt\tau A(\beta)\int_0^\tau s^{-1/2}(\rho_0-\rho-\beta s)^{-\gamma}ds\\
&+\lambda\sqrt{\nu \tau} A(\beta)+\nu\sqrt{\nu \tau}A(\beta) \int_0^\tau (\rho_0-\rho-\gamma s)^{-\gamma}ds+\sqrt{\nu \tau}\beta^{-1}A(\beta)\\
&+A(\beta)^2\int_0^\tau \sqrt{\tau}(\tau-s)^{-1/2}ds+\sqrt{\nu \tau}A(\beta)^2 \int_0^\tau \lw(1+(\rho_0-\rho-\beta s)^{-\gamma}\rw)ds\\
&\lesssim M_0+\lambda A(\beta)+\beta^{-1}A(\beta)^2.
\enda 
\]
Next we bound $\|w(\tau)\|_{\mathcal W_\rho^{k+1,1}}$. Again using Propositions \ref{coupled} and \ref{g-bound}, we get
\[\bega
\|w(\tau)\|_{\mathcal W_\rho^{k+1,1}}&\lesssim M_0+\lambda A(\beta)+\lambda \nu A(\beta)\int_0^\tau(\rho_0-\rho-\beta s)^{-1-\gamma}ds\\
&\quad+\lambda \sqrt\nu A(\beta) \int_0^\tau s^{-1/2}(\rho_0-\rho-\beta s)^{-\gamma-1}ds\\
&\quad+A(\beta)^2\int_0^\tau(\rho_0-\rho-\beta s)^{-1-\gamma}ds+\beta^{-1}A(\beta)\\
&\lesssim M_0+\lambda A(\beta)+\lw(\beta^{-1}A(\beta)+\beta^{-1}A(\beta)^2\rw)(\rho_0-\rho-\beta \tau)^{-\gamma}.
\enda
\]
Finally, we bound $\sqrt{\nu \tau}\|w(\tau)|_{z=0}\|_{\mathcal H_\rho^k}$. From Propositions \ref{at-bdr} and \ref{g-bound}, we get 
\[\bega 
\sqrt{\nu \tau}\|w(\tau)|_{z=0}\|_{\mathcal H_\rho^k}&\lesssim M_0+\lambda \nu\sqrt{\nu \tau} A(\beta)\int_0^\tau s^{-1/2}(\rho_0-\rho-\beta s)^{-\gamma-1}ds\\
&\quad+\lambda A(\beta)+\nu^{3/2}\sqrt \tau A(\beta)\int_0^\tau (\rho_0-\rho-\beta s)^{-\gamma-1}ds\\
&\quad+\sqrt{\nu \tau}\beta^{-1}A(\beta)^2+A(\beta)^2\int_0^\tau \frac{\sqrt \tau}{\sqrt{\tau-s}}(\rho_0-\rho-\beta s)^{-\gamma}ds\\
&\quad+\sqrt{\nu \tau}A(\beta)^2\int_0^\tau (\rho_0-\rho-\beta s)^{-1-\gamma}ds+\sqrt{\nu \tau}\int_0^\tau A(\beta)^2(1+(\rho_0-\rho-\beta s)^{-\gamma})ds\\
&\lesssim M_0+\lambda A(\beta)+\beta^{-1}(A(\beta)+A(\beta)^2)(\rho_0-\rho-\beta \tau)^{-\gamma}.
\enda
\]
Finally, for $\|y^2D_{x,y}^5w\|_{L^2(y\ge \delta_0/2)}$, this is bounded by the functional energy $\mathcal E(t)$ in section \ref{energy-sec}. From Proposition \ref{energy} and Proposition \ref{Na}, we get 
\[
\mathcal E'(\tau)\le C_0 \lw(\mathcal E(\tau)+A(\beta)\mathcal E(\tau)+A(\beta)^2+A(\beta)^2\mathcal E(\tau)^{1/2} \rw).
\]
By Gronwall lemma, we get
\[
 \mathcal E(\tau )\le e^{C_0(1+A(\beta))\tau }\lw(\mathcal E(0)+C_0\int_0^{\tau }A(\beta)^2ds\rw)
\]
Hence 
\[
\|y^2D_{x,y}^5w\|_{L^2(y\ge \delta_0/2)}\le e^{C_0(1+A(\beta))\beta^{-1}}\lw(\|y^2D_{x,y}^5w_0\|_{L^2(y\ge \delta_0/4)}+C_0\beta^{-1}A(\beta)^2\rw).
\]
This completes the proof.
\end{proof}
\section{Proof of the main theorem}
Taking $\beta$ sufficiently large in Proposition \ref{iter-beta}, 
we have $A(\beta)\le C_0$ for some constant $C_0$ that only depends on the size of the initial data. 
This implies 
\[
\|w(\tau)\|_{\mathcal W_\rho^{k,1}}+\sqrt{\nu \tau}\|w(\tau)\|_{\mathcal H_\rho^{k}}\le C_0
\]
uniformly in the time interval $\tau\in \lw[0,\frac{\rho_0}{2\beta}\rw]$.
This implies
\[
\sup_{0\le \tau\le \frac{\rho_0}{2\beta}} \quad \sum_{\al\in \lambda \Z} e^{\eps_0\delta_0|\al|}|w_\al(\tau)|_{z=0}|\le C_0.
\]
To show the uniform bound \eqref{vorticity-bd} on the vorticity, it is natural to switch back to the original variables $(t,\theta,r)$. Using the relation \eqref{w-omega}, we obtain 
\[
\sup_{0\le t\le\frac{\lambda^2\rho_0}{2\beta}}\quad \sqrt{\nu t}\sum_{n\in \mathbb Z} e^{\delta_0\eps_0\lambda |n|}|\w_n(t)|_{z=0}|\le C_0
\]
Let $T=\frac{\lambda^2\rho_0}{2\beta}$. We have, for any $\theta\in \mathbb T$ and $t\in [0,T]$:
\[
|\w^\nu(t,\theta,1)|\le \sum_{n\in \Z}|\w_n(t,1)|\le C_0(\nu t)^{-1/2}\sum_{n\in \Z}e^{-\delta_0\eps_0\lambda |n|} .
\]
Hence we obtain, for some constant $C_0>0$:
\beq\label{time}
\|\w^\nu(t,\theta,r=1)\|_{L^\infty(\mathbb T)}\le C_0 (\nu t)^{-1/2}
\eeq
for all $0\le t\le T$. The proof of \eqref{vorticity-bd} is complete.
To justify the inviscid limit \eqref{inv-lim}, we check the condition 
\[
\nu \int_0^T|\w^\nu(t,\theta,1)|dt\to 0\qquad \text{as}\quad \nu \to 0.
\]
This is direct from the bound \eqref{time}. The proof of Theorem \ref{theo-main} is complete.

%

\def\cprime{$'$} \def\cprime{$'$}

\end{document}